\documentclass[11pt,leqno]{amsart}

\usepackage[english]{babel}
\usepackage[margin=1.15in]{geometry}

\usepackage{amsfonts,amssymb}
\usepackage{amsmath, amsthm, amsfonts}
\usepackage{amsrefs}

\usepackage{color}
\usepackage{hyperref}
\usepackage{enumitem}  

\theoremstyle{plain}

\newtheorem*{theorem*}{Theorem}
\newtheorem{theorem}{Theorem}[section]
\newtheorem{proposition}[theorem]{Proposition}
\newtheorem{corollary}[theorem]{Corollary}
\newtheorem{lemma}[theorem]{Lemma}

\theoremstyle{definition}
\newtheorem{definition}[theorem]{\scshape{Definition}}

\newtheorem{example}[theorem]{\scshape{Example}}

\newtheorem{remark}[theorem]{\scshape{Remark}}

\DeclareMathOperator{\Aut}{Aut}
\DeclareMathOperator{\Sym}{Sym}
\DeclareMathOperator{\St}{St}
\DeclareMathOperator{\rk}{rk}

\numberwithin{equation}{section}

\author[B. Klopsch]{Benjamin Klopsch} \address{Benjamin Klopsch:
  Heinrich-Heine-Universit\"at, Mathematisch-Naturwissenschaftliche
  Fakult\"at, Mathematisches Institut, D\"usseldorf, Germany}
\email{klopsch@math.uni-duesseldorf.de}

\author[A. Thillaisundaram]{Anitha Thillaisundaram} \address{Anitha
  Thillaisundaram: Centre for Mathematical Sciences, Lund University,
  223 62 Lund, Sweden} \email{anitha.thillaisundaram@math.lu.se}

\date{\today}

\thanks{This research was supported by the Knut and Alice Wallenberg
  Foundation and by the Royal Physiographic Society of Lund. The second author also acknowledges support from the  Folke Lann\'{e}r's Fund.}

\title[Normal subgroups of non-torsion multi-EGS groups]{Normal
  subgroups of non-torsion multi-EGS groups}

\begin{document}

\begin{abstract}
  We study the distribution of normal subgroups in non-torsion, regular
  branch multi-EGS groups and show that the congruence completions of
  such groups have bounded finite central width.  In particular, we
  show that the profinite completion of the Fabrykowski--Gupta group
  acting on the $p$-adic tree has central width $2$ for every odd
  prime~$p$.  The methods used also apply to the family of \v{S}uni\'{c}
  groups, which closely resemble the Grigorchuk group.
\end{abstract}

\maketitle


\section{Introduction}

Regular branch groups are infinite groups acting on regular rooted
trees that feature a fractal-like subgroup structure.  They first
appeared in the 1980s, for instance in relation to the Burnside
problem, and gave rise to explicit examples of finitely generated
infinite torsion groups, finitely generated groups of intermediate
word growth and finitely generated amenable but not elementary
amenable groups.  Since then, the theory of regular branch groups has
developed extensively and nowadays features applications within group
theory and to other areas, such as to dynamics, analysis and geometry.
Regular branch groups are part of the larger family of branch groups;
however, we will not consider these more general groups here.  The
basic definitions of regular branch groups and related notions are
collected in Section~\ref{sec:prelim}; further information can be
found in~\cite{BarthGrigSunik}.

For a prime $p$, the $p$-adic tree $T$ is the infinite regular rooted
tree where each vertex has $p$ descendants.  The first examples of
regular branch groups were constructed, by Grigorchuk~\cite{Gr80} and
by Gupta and Sidki~\cite{GuSi83}, as subgroups of the automorphism
groups $\Aut T$ for such trees~$T$.  Their pioneering work soon led to
the notion of Grigorchuk--Gupta--Sidki groups, or GGS-groups for
short.  A GGS-group $G = \langle a, b \rangle$ is a $2$-generated
subgroup of $\Aut T$, for an odd prime~$p$, such that the `rooted'
generator $a$ cyclically permutes the $p$ first-level vertices and the
`directed' generator $b$ is recursively defined along an infinite
directed path of the tree; see Section~\ref{sec:GGS} for details.
Several generalisations of GGS-groups have been studied, one of which
is the family of multi-EGS groups, where EGS stands for `extended
Gupta--Sidki'. A multi-EGS group is, simply put, a group generated by
the rooted automorphism $a$ and several directed automorphisms,
each given by a direction and a defining vector, where
some of the directed generators are allowed to be defined
over different directed paths; see Section~\ref{sec:multi-EGS}.

The group $\Aut T$ carries a natural congruence topology, turning it
into a totally disconnected, compact topological group.  A subgroup
$G \le \Aut T$ inherits the congruence topology, which can be
described in a concrete way as follows.  For $n\in\mathbb{N}_0$, the
$n$th level stabiliser $\St_G(n)$, also termed the $n$th principal
congruence subgroup of~$G$, consists of all elements of $G$ that
pointwise fix all $n$th-level vertices.  The cosets of these principal
congruence subgroups form a base for the congruence topology on~$G$.
A (finite-index) subgroup $H \le G$ is a congruence subgroup if it
contains $\St_G(n)$ for some $n \in \mathbb{N}$, i.e.\! if it is open
in the congruence topology.  The group $G\le \Aut T$ has the
congruence subgroup property if every finite-index subgroup is a
congruence subgroup.

While some aspects of regular branch groups have been thoroughly
investigated, we still have many open questions about the distribution
and properties of their normal subgroups.  For specific groups and
special normal subgroups, such as terms of the lower central series or
the derived series, the unfolding picture is very interesting.  For
instance, Vieira~\cite{Vi98} determined the derived series of the
Gupta--Sidki $3$-group and obtained partial results on the lower
central series of that group.  Based on computational data, Bartholdi,
Eick and Hartung~\cite{BaEiHa08} and Hartung~\cite{Ha13} established
partial results concerning the lower central series for several
regular branch groups, and also weakly regular branch
groups. Petschick~\cite{PeXX} recently determined the derived series
of all regular branch GGS-groups.  Regarding general normal subgroups,
Ceccherini-Silberstein, Scarabotti and Tolli attained an effective
version of the congruence subgroup property for the Grigorchuk
group~$\mathfrak{G}$ (the first group constructed by Grigorchuk
in~\cite{Gr80}).  They showed that for non-trivial normal
subgroups $N\trianglelefteq \mathfrak{G}$, if $m$ is maximal such that
$N\subseteq\St_{\mathfrak{G}}(m)$, then
$\St_{\mathfrak{G}}(m+3)\subseteq N$; see \cite[Cor.~5.13]{CeScTo01}.
Based on this result, they explicitly described all normal subgroups
of the Grigorchuk group that are not contained in
$\St_{\mathfrak{G}}(4)$.  Subsequently, Bartholdi~\cite{Ba05}
described an explicit scheme for pinning down all normal subgroups of
the Grigorchuk group and thereby observed that every proper normal
subgroup of the Grigorchuk group can be normally generated by at most
$2$ elements.  He also improved on existing results concerning the
derived series and the lower central series for two well-studied
GGS-groups: the Gupta--Sidki $3$-group and the Fabrykowski--Gupta
group acting on the $3$-adic tree; see \cite[Thm.~3.12 and
Thm.~3.15]{Ba05}.

\smallskip

In this paper, we study normal congruence subgroups and their
distribution in regular branch multi-EGS groups, acting on the
$p$-adic tree~$T$ for $p$ an odd prime, with a focus on non-torsion
groups.  Akin to~\cite{CeScTo01}, we establish an effective version of
the congruence subgroup property.

\begin{theorem}\label{thm:sharper-CSP}
  Let $G \le \Aut T$ be a multi-EGS group and let
  $N \trianglelefteq G$.  Suppose that $[N,G]$, hence also
    $N$, is a congruence subgroup, and let $m\in \mathbb{N}_0$ be maximal subject to
  $N\subseteq \St_G(m)$.  Then the following hold:
  \begin{enumerate}
  \item [\textup{(i)}] if $G$ is regular branch over $[G,G]$ then
    $\St_G(m + \dot{r}_G + 3)\subseteq [N,G]$, where
    $\dot{r}_G$ denotes the maximal number of linearly
      independent defining vectors for the directed automorphisms in a
      standard generating system for~$G$;
  \item [\textup{(ii)}] if $G$ is regular branch over $\gamma_3(G)$
    but not over $[G,G]$, then $\St_G(m+7)\subseteq [N,G]$.
  \end{enumerate}
  
  In the special case where $G$ is a GGS-group, the conclusion in
  \textup{(i)} improves to $\St_G(m+3) \subseteq [N,G]$ and the
  conclusion in \textup{(ii)} to $\St_G(m+4) \subseteq [N,G]$.

  Finally, if $G$ is the Fabrykowski--Gupta group for the prime~$p$,
  then $\St_G(m+2)\subseteq [N,G]$.
\end{theorem}

\noindent  We recall that a standard generating system
for a multi-EGS group is in particular a minimal set of generators;
see Section~\ref{sec:multi-EGS}.  Furthermore, the minimal
  number of generators for the congruence completion $\overline{G}$,
  the topological closure of~$G$ within~$\Aut T$, of a multi-EGS group
  $G$ that is regular branch over $[G,G]$ equals $1+\dot{r}_G$;
  see  Corollary~\ref{prop:r-dot}.

\begin{remark}
  In the situation of Theorem~\ref{thm:sharper-CSP}, when $G$ has the
  congruence subgroup property, such as when $G$ is a GGS-group or
  even the Fabrykowski--Gupta group, then the conclusion applies to
  all non-trivial normal subgroups; compare with
  Proposition~\ref{prop:EGS-CSP},which records
    \cite[Thm.~1.1]{ThUr21} modulo a correction.
\end{remark}

A general, but less effective, version of
Theorem~\ref{thm:sharper-CSP} was already known. Specifically, for any
regular branch group~$G$ with the congruence subgroup property, there
exists a uniform bound $k_G\in\mathbb{N}$ such that
$\St_G(m+k_G)\subseteq [N,G]$, where $N$ and $m$ are as in
Theorem~\ref{thm:sharper-CSP}; see Remark~\ref{rmk:CSP-bound}.

The next step towards understanding the normal congruence subgroups of
a multi-EGS group~$G$ as in Theorem~\ref{thm:sharper-CSP} is to
describe the normal subgroups $N \trianglelefteq G$ that are
sandwiched between two consecutive level stabilisers.  In a somewhat
more general setting we obtain the following result; see
Section~\ref{sec:chain} for a detailed analysis providing more
structural information and references to prior related work.

\begin{theorem} \label{thm:strong-sandwich-gives-chain} Let
  $S \le \Aut T$ be the Sylow pro-$p$ subgroup consisting of all
  elements whose labels are powers of~$a$, the rooted $p$-cycle
  permuting transitively the first-level vertices.  Let
  $G = \langle a \rangle \ltimes \St_G(1) \le S$ be a self-similar
  group containing a directed automorphism $b \in \St_S(1)$ such that
  \[
    \psi(b) = (a^{e_1},\dots,a^{e_{p-1}},b) \quad \text{with}
    \quad \sum\nolimits_{i=1}^{p-1} e_i \not\equiv_p 0,
  \]
  where
  $\psi \colon \St_G(1) \to G \times \overset{p}{\dots} \times G$
  denotes the natural embedding. Then, for every $m \in \mathbb{N}$,
  the normal subgroups $N \trianglelefteq G$ with
  $\St_G(m+1) \subseteq N \subseteq \St_G(m)$ form a chain
  \[
  \St_G(m+1) = N_0 \subsetneq 
  N_1 \subsetneq 
  \cdots \subsetneq
  N_{t(m)} = \St_G(m)  
  \]
  of length $t(m) = \log_p \lvert \St_G(m) : \St_G(m+1) \rvert$ so
  that $\lvert N_j : N_{j-1} \rvert = p$ for $1 \le j \le t(m)$.

  Furthermore, if $G$ is a multi-EGS group, then the normal subgroups
  $N_0, N_1, \dots, N_{t(m)}$ are also characteristic in~$G$, and
  $t(m) \le p t(m+1)$ for every $m$. Additionally if
  $G = \langle a,b \rangle$ is a (non-torsion) GGS-group and regular
  branch over $[G,G]$, then $t(1)=p$ and $t(m) = (p-1) p^{m-1}$ for
  $m \ge 2$.
\end{theorem}

In the setting of regular branch GGS-groups,
Theorem~\ref{thm:sharper-CSP} and
Theorem~\ref{thm:strong-sandwich-gives-chain} suggest that, with some
extra work, it is feasible to obtain a complete description of the
distribution of all normal congruence subgroups.  Indeed, we intend to
do so in a future work, for groups that are similar to the
  Fabrykowski--Gupta group.

Next we turn towards some structural properties of normal subgroups of
non-torsion multi-EGS groups.  As a consequence of
Theorems~\ref{thm:sharper-CSP}
and~\ref{thm:strong-sandwich-gives-chain} we obtain bounds for the
numbers of normal generators.  For any group $G$, let
$\rk^{\trianglelefteq}(G)$ denote the \emph{normal rank} of $G$, i.e.
\[
  \rk^{\trianglelefteq}(G) = \sup \{ d_G^{\trianglelefteq}(N) \mid N
  \trianglelefteq G \text{ with } d_G^{\trianglelefteq}(N) < \infty \}
  \in \mathbb{N}_0 \cup \{ \infty \},
\]
where $d_G^{\trianglelefteq}(N)$ denotes the minimal number of normal
generators of $N \trianglelefteq G$.

\begin{corollary}\label{cor:normal-generation}
  Let $G \le \Aut T$ be a non-torsion multi-EGS group with the
  congruence subgroup property, and let $r_G$ denote the number of
  directed automorphisms in a standard generating system for~$G$.
  Then the normal rank of $G$ is bounded as follows:
  \[
    \rk^{\trianglelefteq}(G) \le
  \begin{cases}
    r_G + 3 & \text{if $G$ is regular branch over $[G,G]$,} \\
    7 & \text{if $G$ is regular branch over $\gamma_3(G)$ but not
      over $[G,G]$,} \\
    3 & \text{if $G$ is a GGS-group and regular branch over
      $[G,G]$,} \\
    4 & \text{if $G$ is a GGS-group and regular branch over
      $\gamma_3(G)$ but not over $[G,G]$.}
  \end{cases}
  \]
  If $G$ is the Fabrykowski--Gupta group for the prime $p$, then
  $\rk^{\trianglelefteq}(G) = 2$.
\end{corollary}

\noindent We remark that, if $G$ is a multi-EGS group that
  is regular branch over $[G,G]$ and has the congruence subgroup
  property, then $\dot r_G$ equals $r_G$; see Proposition~\ref{prop:EGS-CSP}.

\smallskip

Finally, we apply our results to bound the central width of a
non-torsion, regular branch
multi-EGS group~$G$, more precisely of its congruence
completion~$\overline{G}$.  We observe that $\overline{G}$ is a
finitely generated just infinite pro-$p$ group; see
Section~\ref{sec:GGS}.  The \emph{width} of such a pro-$p$ group
$\Gamma$ is defined as
\[
  w(\Gamma) = \sup \{ \log_p \lvert
  \gamma_n(\Gamma):\gamma_{n+1}(\Gamma) \rvert \mid n \in \mathbb{N}
  \} \in \mathbb{N}_0 \cup \{ \infty \};
\]
its use is to generalise the concept of finite coclass.  Linear pro-$p$ groups of finite width were studied in~\cite{KlLePl97}, and special interest has been shown in finding other explicit examples of just infinite pro-$p$ groups of finite width; for instance, see~\cite{Gr05} and the references therein.  We consider the \emph{central  width} of a pro-$p$ group $\Gamma$, defined as
\[
  w_\mathrm{cen}(\Gamma) = \sup \{ \log_p \lvert \Delta :
  [\Delta,\Gamma] \rvert \mid \Delta \trianglelefteq_\mathrm{o} \Gamma
  \} = \sup \{ \log_p \lvert \Delta :
  [\Delta,\Gamma] \rvert \mid \Delta \trianglelefteq_\mathrm{c} \Gamma
  \} \in \mathbb{N}_0 \cup \{ \infty \};
\]
compare with~\cite[I~b)]{KlLePl97} and \cite{CaCa01}.  Clearly,
$w(\Gamma) \le w_\mathrm{cen}(\Gamma)$ so that upper bounds for
$w_\mathrm{cen}(\Gamma)$ also yield corresponding bounds for~$w(\Gamma)$.

Bartholdi and Grigorchuk computed the lower central series of the
Grigorchuk group and the Grigorchuk overgroup, and their results show
that the completion of the group has finite width $3$ and $4$
respectively; see \cite[Thm.~6.4 and Thm.~7.4]{BaGr00}.  By a detailed
study of the graded Lie algebra associated to the lower central
series, Bartholdi showed that the completion of the Gupta--Sidki
$3$-group has infinite width and that the completion of the
Fabrykowski--Gupta group acting on the $3$-adic tree has width~$2$;
see~\cite[Cor.~3.9 and Cor.~3.14]{Ba05} and \cite[Thm.~16]{BaEiHa08}.
With considerable less effort we obtain the following general bounds.

\begin{corollary}\label{cor:central-width}
  Let $G \le \Aut T$ be a non-torsion multi-EGS group, and let
  $\dot{r}_G$ denote the maximal number of linearly independent
    defining vectors for the directed automorphisms
    in a standard generating system for~$G$.  Then the central width of
  the congruence completion $\overline{G}$ is bounded as
  follows:
  \[
    w_\mathrm{cen}(\overline{G}) \le
    \begin{cases}
     \dot{r}_G + 3 & \!\text{if $G$ is regular branch over $[G,G]$,} \\
      7 & \!\text{if $G$ is regular branch over $\gamma_3(G)$ but not
        over $[G,G]$,} \\
      3 & \!\text{if
        $G$ is a GGS-group and regular branch over $[G,G]$,} \\
      4 & \!\text{if $G$ is a GGS-group and regular branch over
        $\gamma_3(G)$ but not over $[G,G]$.}
    \end{cases}
  \]
  If $G$ is the Fabrykowski--Gupta group for the prime $p$, then
  $w_\mathrm{cen}(\overline{G}) = 2$.
\end{corollary}

The last assertion settles a conjecture of Bartholdi, Eick and
Hartung~\cite[Conj.~17]{BaEiHa08}.  The conjecture was independently
proved by Fern\'{a}ndez-Alcober, Garciarena and Noce~\cite{Mikel}, who
give a detailed description of the lower central series of the
Fabrykowski--Gupta group, and more generally, of GGS-groups of
FG-type.  Computational evidence also indicates that the bound $3$
above is best possible for GGS-groups~$G$ that are regular branch over
$[G,G]$, but not of FG-type; see \cite{Mikel} for the definition of
groups of FG-type. This also reflects on the sharpness of
the corresponding bounds in Theorem~\ref{thm:sharper-CSP} and
Corollary~\ref{cor:normal-generation}.

\smallskip

Finally, our methods also apply to the family of \v{S}uni\'{c} groups,
which closely resemble the Grigorchuk group. To not disrupt the flow
of the paper, we refer the reader to Appendix~\ref{sec:Sunic} for
details of these groups, relevant notation, and for the proofs of all
corresponding results.

\begin{theorem}\label{thm:Sunic}
  Let $G\le \Aut T$ be a regular branch \v{S}uni\'{c} group acting on
  the $p$-adic tree~$T$, where $p$ is any prime, and let $r_G$ denote
  the number of directed automorphisms in a standard generating system
  for~$G$. For $p=2$, let $n_G$ be the parameter as defined in
  Proposition~\ref{prop:Sunic-CSP}.  Then the normal rank of $G$ and
  the central width of its congruence completion~$\overline{G}$ are
  bounded as follows:
  \[
   \rk^{\trianglelefteq}(G) ,  w_\mathrm{cen}(\overline{G}) \le
   \begin{cases}
     r_G+3 & \text{if $p$ is odd,} \\
     r_G+n_G+3 & \text{if $p=2$.}
    \end{cases}
  \]
\end{theorem}

Unlike our previous results for multi-EGS groups, the results for the
\v{S}uni\'{c} groups include torsion groups.

 \bigskip
 
 We conclude with the observation that we obtain
 infinitely many profinite isomorphism classes of groups with finite
 central width.
 
 \begin{corollary}\label{cor:EGS-profinitely-non-isomorphic}
   For each prime $p\ge 3$, there are at least two non-torsion
   multi-EGS groups $G$ and $H$, with non-isomorphic profinite
   completions, each of finite central width.
 \end{corollary}

It is of independent interest to determine under what
 circumstances two multi-EGS groups are profinitely isomorphic.  In
 Remark~\ref{rem:not-iso-but-prof-iso} we collect some examples of
 non-isomorphic but profinitely isomorphic multi-EGS groups.

\smallskip


\noindent\textit{Organisation}. Section~\ref{sec:prelim} contains
preliminary material on regular branch groups. In
Section~\ref{sec:groups}, we formally define the GGS-groups and
multi-EGS groups, and we state some of their basic properties. In
Section~\ref{sec:sandwich} we prove Theorem~\ref{thm:sharper-CSP} and
in Section~\ref{sec:chain} we prove
Theorem~\ref{thm:strong-sandwich-gives-chain}. Finally, in
Section~\ref{sec:width} we prove our remaining results concerning
multi-EGS groups, before ending with  
Appendix~\ref{sec:Sunic} which concerns the \v{S}uni\'{c} groups.

\medskip

\noindent\textit{Notation.}
The set of positive integers is denoted
  by $\mathbb{N}$ and the set of non-negative integers
  by~$\mathbb{N}_0$.  We write $\mathbb{F}_p =
\mathbb{Z}/p\mathbb{Z}$ for the finite field with $p$ elements.  The
terms of the lower central series of a group $G$ are denoted by
$\gamma_i(G)$, $i \in \mathbb{N}$.  We write
$G' = [G,G] = \gamma_2(G)$ for the commutator subgroup.  Throughout we
use left-normed commutators, e.g., $[x,y,z] = [[x,y],z]$.

\medskip

\noindent\textbf{Acknowledgements.} We are grateful to Gustavo
Fern\'{a}ndez-Alcober and Marialaura Noce for their helpful comments,
and to Mikel Garciarena for his help with GAP.


\section{Preliminaries} \label{sec:prelim}

Here we recall the notion of regular branch groups and related
notions.  We establish some prerequisites and notation for the rest of
the paper.  For more information, see~\cite{BarthGrigSunik}.

\subsection{The \texorpdfstring{$p$}{p}-adic tree and its automorphisms}
Let $p$ be a prime and let $T$ be the \emph{$p$-adic tree}, that is,
an infinite regular rooted tree where every vertex has $p$
descendants.  Taking $X = \{1,2,\dots,p\}$ as an alphabet on $p$
letters, the set of vertices of~$T$ can be identified with the free
monoid~$X^*$.  In accordance with this identification, the root of~$T$
is the empty word~$\varnothing$, and for each word $v\in X^*$ and
letter~$x\in X$, there is an edge connecting $v$ to $vx$.  There is a
natural length function $\lvert \cdot \rvert$ on~$X^*$ which is in
line with the combinatorial distance between vertices of~$T$.  The
vertices that are at distance~$n$ from the root form the \emph{$n$th
  layer} of the tree.  The \emph{boundary}~$\partial T$ consists of
all infinite simple rooted paths and is naturally in one-to-one
correspondence with the $p$-adic integers.

For a vertex $u$ of $T$, we write $T_u$ for the full rooted subtree
of~$T$ that has its root at~$u$, so $T_u$ includes all vertices $v$
with $u$ a prefix of~$v$. For any two vertices $u$ and $v$ the
subtrees $T_u$ and $T_v$ are isomorphic under the map that deletes the
prefix $u$ and replaces it by the prefix~$v$.  Using this natural
identification of subtrees, we can describe induced actions of
automorphisms on subtrees in terms of automorphisms of $T$ itself, as
follows.

Every automorphism of~$T$ must fix the root, and the orbits of
$\Aut T$ on~$T$ are precisely its layers.  For $f \in \Aut T$, the
image of a vertex $u$ under $f$ will be denoted by~$u^f$.  For a
vertex~$u$, considered as a word over $X$, and a letter $x \in X$ we
have $(ux)^f=u^fx'$ where $x' \in X$ is uniquely determined by $u$
and~$f$.  This yields a permutation $f(u)\in \text{Sym}(X)$ satisfying
\[
(ux)^f = u^f x^{f(u)}.
\]
We refer to the permutation~$f(u)$ as the \emph{label} of~$f$
at~$u$. An automorphism $f$ is called \emph{rooted} if $f(u)=1$ for
$u\ne\varnothing$.  An automorphism~$f$ is called \emph{directed},
with directed path $\ell$ for some $\ell\in \partial T$, if the
support $\{u \mid f(u)\ne1 \}$ of its labels is infinite and contains
only vertices at distance $1$ from~$\ell$.  The \emph{section} of
$f$ at a vertex $u$ is the unique automorphism $f_u \in \Aut T$ given
by the condition $(uv)^f = u^f v^{f_u}$ for $v \in X^*$.

\subsection{Notable subgroups of \texorpdfstring{$\Aut T$}{Aut T}}\label{subsec:CSP}
Let $G\le \Aut T$. For a vertex $u$, the \emph{vertex stabiliser}
$\mathrm{st}_G(u)$ is the subgroup consisting of all elements in~$G$ that
fix~$u$.  For $n \in \mathbb{N}_0$, the \emph{$n$th level stabiliser}
is the normal subgroup
$\St_G(n)= \cap_{|v|=n} \mathrm{st}_G(v) \trianglelefteq G$.  The full automorphism group $\Aut T$ is a profinite group, with the subgroups $\St_{\Aut T}(n)$, for $n \in \mathbb{N}$, providing a base
of open neighbourhoods for the identity element.  A \emph{congruence
  subgroup} of $G$ is a subgroup $H \le G$ such that
$\St_G(n) \subseteq H$ for some $n \in \mathbb{N}$.  The group
$G\le \Aut T$ has the \emph{congruence subgroup property} if every
finite-index subgroup of $G$ is a congruence subgroup, equivalently if its topological closure $\overline{G}$ in $\Aut T$ yields the
profinite completion of~$G$.

For $n\in \mathbb{N}$, every element $g \in \St_{\Aut T} (n)$ is
determined by its sections at the $n$th level vertices, i.e.\! a
collection $g_1,\dots,g_{p^n}$ of $p^n$ elements of $\Aut T$.
Denoting the vertices of $T$ at level~$n$ by $u_1, \dots, u_{p^n}$,
we obtain a natural embedding
\[
  \psi_n \colon \St_{\Aut T}(n) \longrightarrow
  \prod\nolimits_{i=1}^{p^n} \Aut T_{u_i} \cong \Aut T \times
  \overset{p^n}{\cdots} \times \Aut T, \quad g \mapsto
  (g_1,\dots,g_{p^n}).
\]
For convenience, we will write $\psi=\psi_1$.  For a vertex $u$, we
further write
\[
  \varphi_u \colon \mathrm{st}_{\Aut T}(u) \longrightarrow \Aut T_u
  \cong \Aut T, \quad f \mapsto f_u
\]
for the natural restriction of $f$ to its section $f_u$.

A group $G \le \Aut T$ is \emph{spherically transitive} if it acts
transitively on every layer of~$T$.  The group $G$ is
\emph{self-similar} if $\varphi_u(\mathrm{st}_G(u))\subseteq G$ for
every vertex~$u$, and $G$ is \emph{super strongly fractal} if for
every $n\in\mathbb{N}$ and every $n$th-level vertex~$u$ we have
$\varphi_u(\St_G(n)) = G$.  The group $G$ is said to be \emph{regular
  branch} over a finite-index subgroup $K \le G$, if (i) $G$ is
spherically transitive, (ii) $G$ is self-similar and (iii)
$K\times \overset{p}\dots \times K \subseteq \psi(\St_K(1))$.  We observe that,
if $G$ is regular branch over~$K$, then
$\lvert G :\psi_n^{-1}(K\times \overset{p^n}\dots \times K) \rvert <
\infty$ for all $n\in\mathbb{N}$.


\section{Multi-EGS groups}\label{sec:groups}

Here we recall briefly the notion and basic properties of multi-EGS
groups.  We begin our discussion with GGS- and multi-GGS groups, which
are two important special classes of multi-EGS groups.  The technical
set-up for these groups is less complicated, and we make use of them
to deal with more general multi-EGS groups.  As for the
  rest of the paper, excluding Appendix~\ref{sec:Sunic}, the prime~$p$
  is odd and all groups considered here are subgroups of the
automorphism group $\Aut T$ of the $p$-adic tree~$T$.

\subsection{GGS-groups and multi-GGS groups} \label{sec:GGS} We denote
by $a$ the rooted automorphism corresponding to the $p$-cycle
$(1 \, 2 \, \cdots \, p)\in \text{Sym}(p)$ that cyclically permutes
the $p$ vertices forming the first layer of~$T$.  Given a vector
$ \mathbf{e} =(e_{1}, e_{2},\dots , e_{p-1})\in
(\mathbb{F}_p)^{p-1}\backslash \{\mathbf{0}\}$, a corresponding
directed automorphism $b \in \St_{\Aut T}(1)$ is recursively defined
via
\[
\psi(b)=(a^{e_{1}}, a^{e_{2}},\dots,a^{e_{p-1}},b).
\]
Then $G_{\mathbf{e}}=\langle a, b \rangle$ is the \emph{GGS-group}
associated to the \emph{defining vector} $\mathbf{e}$.
The vector $\mathbf{e}$ is said to be \emph{symmetric} if
$e_{i}=e_{p-i}$ for $i\in\{1,\dots, \frac{p-1}{2}\}$, and
\emph{non-symmetric} otherwise.

By definition $\langle a\rangle \cong\langle b \rangle \cong C_p$ are
cyclic of order~$p$. The GGS-group $G_{\mathbf{e}}$ is a torsion group, and thus an
infinite $p$-group, if and only if
$ \sum\nolimits_{j=1}^{p-1} e_{j}\equiv_p 0$; compare \cite{Vo00}.  We
write $\mathcal{G} = \langle a,b \rangle$ with
$\psi(b) = (a,\dots,a,b)$, for the GGS-group arising from the
constant defining vector~$(1,\dots,1)$ or, indeed, any other constant
non-zero vector.  It is known that $\mathcal{G}$ is not regular
branch; see \cite[Lem.~4.2]{FeZu14} and \cite[Thm.~3.7]{FeGaUr17}.  We
recall a number of basic facts.

\begin{proposition}\label{prop:GGS-facts}\cite[Lem.~3.2, 3.3, 3.5 and Cor.~2.5]{FeZu14}
  Let $G = G_{\mathbf{e}}$ be a GGS-group.  Then the following hold:
  \begin{enumerate}
  \item [\textup{(i)}] if $\mathbf{e}$ is non-symmetric, then $G$ is
    regular branch over $[G,G]$;
  \item [\textup{(ii)}] if $\mathbf{e}$ is non-constant and symmetric,
    then $G$ is regular branch over $\gamma_3(G)$ but not over
    $[G,G]$;
  \item [\textup{(iii)}] $\St_G(2)\subseteq \gamma_3(G)$.
  \end{enumerate}
\end{proposition}

\begin{proposition}\label{prop:GGS-rank-p}\cite[Thm.~2.4(i), Thm.~2.14
  and Lem.~3.4]{FeZu14} Let $G$ be a non-torsion GGS-group that is
  regular branch over $G' = [G,G]$. Then
  $\St_G(2)=\St_G(1)'=\psi^{-1}(G'\times\overset{p}\dots\times G')$.
\end{proposition}

Two well-studied GGS-groups include the Gupta--Sidki $p$-group, which
has defining vector $\mathbf{e}=(1,-1,0,\dots,0)$, and the
Fabrykowski--Gupta group for the prime~$p$, defined by
$\mathbf{e}=(1,0,\dots,0)$.  Traditionally, the Fabrykowski--Gupta
group was only considered for $p=3$, but we use the name more
generally to refer to the corresponding group for all $p \ge 3$.

\smallskip

A straightforward generalisation of the GGS-groups is the family of
multi-GGS groups, which is defined as follows. Given
$r \in\{1,\dots, p-1\}$ and a finite $r$-tuple $\mathbf{E}$ of
$\mathbb{F}_p$-linearly independent vectors
\[
\mathbf{e}_i =(e_{i,1}, e_{i,2},\dots , e_{i,p-1})\in
\mathbb{F}_p^{\; p-1}, \qquad i\in \{1,\dots,r \},
\]
the directed automorphisms $b_1, \dots, b_r \in \St_{\Aut T}(1)$ are
recursively defined via
\[
  \psi(b_i)=(a^{e_{i,1}}, a^{e_{i,2}},\dots,a^{e_{i,p-1}},b_i),
  \qquad i\in \{1,\dots,r \}.
\]
The group $G_{\mathbf{E}}=\langle a, b_1, \dots, b_r \rangle$ is
called the \emph{multi-GGS group} associated to the defining vector
system $\mathbf{E}$. For $r=1$ we simply recover the previous notion
of a GGS-group. Properties of multi-GGS groups will be collected among
the results of multi-EGS groups below.

\subsection{Multi-EGS groups}\label{sec:multi-EGS}
Let $r_1, \dots, r_p \in \{0,1,\dots,p-1\}$, with $r_j \not = 0$ for
at least one index~$j$, and set $r=r_1+\cdots +r_p$.  Let
$\mathbf{E} = (\mathbf{E}^{(1)}, \dots, \mathbf{E}^{(p)})$ be a
collection of vector systems
$\mathbf{E}^{(j)} = (\mathbf{e}^{(j)}_1, \dots,
\mathbf{e}^{(j)}_{r_j})$, each consisting of $\mathbb{F}_p$-linearly
independent vectors
\[
  \mathbf{e}^{(j)}_i = \big( e^{(j)}_{i,1}, \dots, e^{(j)}_{i,p-1}
  \big) \in (\mathbb{F}_p)^{p-1}, \qquad i \in \{1, \dots, r_j \}.
\]
The \emph{multi-EGS group} associated to $\mathbf{E}$ is the group
\begin{equation}\label{equ:def-multi-EGS}
  G = G_\mathbf{E} = \langle a, \mathbf{b}^{(1)}, \dots,
  \mathbf{b}^{(p)} \rangle = \big\langle \{a \} \cup \{ b^{(j)}_i \mid
  1 \leq j \leq p, \, 1 \leq i \leq r_j \} \big\rangle,
\end{equation}
where, for each $j \in \{1,\dots,p\}$, the generator family
$\mathbf{b}^{(j)} = \{b^{(j)}_1, \dots, b^{(j)}_{r_j}\}$ consists of
commuting directed automorphisms
$b^{(j)}_i \in \mathrm{St}_{\Aut T}(1)$ recursively defined along the
directed path
\[
\big( \varnothing, (p-j+1), (p-j+1)(p-j+1), \dots \big) \in \partial T
\]
as
\[
  \psi(b^{(j)}_i) = \Big( a^{e^{(j)}_{i,j}}, \dots,
  a^{e^{(j)}_{i,p-1}},b^{(j)}_i, a^{e^{(j)}_{i,1}},\dots,
  a^{e^{(j)}_{i,j-1}} \Big);
\]
we refer to the vector $\mathbf{e}^{(j)}_i$ as the defining vector of
$b^{(j)}_i$. Additionally, we refer to the set
$\{a \} \cup \{ b^{(j)}_i \mid 1 \leq j \leq p, \, 1 \leq i \leq r_j
\} $ as a standard generating system for $G$.

Writing $G' = [G,G]$ for the commutator subgroup of
$G = G_\mathbf{E}$, we see from \cite[Prop.~3.9]{KlTh18} that
\begin{equation*}\label{eq:abelianisation2}
  G/G'=\langle aG',\, b_1^{(1)}G',\, \dots,\, b_{r_1}^{(1)}G',\,\,
  \dots,\,\, b_1^{(p)}G',\, \dots,\, b_{r_p}^{(p)}G' \rangle\cong C_p^{\,r+1}.
\end{equation*}
In particular, the number $r$ of directed automorphisms in a standard
generating system for $G$ is intrinsic to the group.

\begin{definition}\label{def:standard-gens}
  We denote by $r_G$ the total number of directed automorphisms in a
  standard generating system for the multi-EGS group $G$.
\end{definition}

Furthermore, the multi-EGS group~$G=G_\mathbf{E}$ is a torsion group,
and thus an infinite $p$-group, if and only if
$\sum\nolimits_{k=1}^{p-1} e_{i,k}^{(j)} \equiv_p 0$ for all
$j \in \{1,\dots,p\}$ and $i \in \{1, \dots, r_j\}$;
compare~\cite[Lem.~3.13]{ThUr21}.  The multi-EGS group $G$ is just
infinite, unless $G = \mathcal{G}$; see \cite[Cor.~1.2]{ThUr21}. We
recall some additional facts, starting with the following dichotomy.

\begin{proposition}\label{prop:EGS-facts}\cite[Prop.~3.2 and 3.5, Lem.~3.3]{ThUr21}
  Let $G = G_\mathbf{E}$ be a multi-EGS group, and let
  $\dot{\mathbf{E}}$ denote the concatenation of the relevant systems
  $\mathbf{E}^{(1)}, \dots, \mathbf{E}^{(p)}$.
  \begin{enumerate}
  \item[\textup{(i)}] If $\dot{\mathbf{E}}$ contains at least one
    non-symmetric vector or at least two linearly independent vectors,
    then $G$ is regular branch over $[G,G]$.
  \item[\textup{(ii)}] If there is a non-constant, symmetric vector
    $\mathbf{e} \in \mathbb{F}_p^{\; p-1}$ such that, for every
    $j \in \{1, \dots,p\}$, the generating family $\mathbf{b}^{(j)}$
    is either empty or consists of a single directed automorphism
    $b_1^{(j)}$ with defining vector
    $\mathbf{e}_1^{(j)} \in \mathbb{F}_p \mathbf{e}$, then $G$ is
    regular branch over~$\gamma_3(G)$ but not over $[G,G]$.
    \end{enumerate}
\end{proposition}

\begin{proposition}\label{prop:EGS-subdirect}\cite[Lem.~3.7]{ThUr21}
  Let $G$ be a multi-EGS group.
  \begin{enumerate}
  \item[\textup{(i)}] If $G$ is regular branch over $[G,G]$, then
    $\psi([G,G])$ is subdirect in $G\times\cdots\times G$.
  \item[\textup{(ii)}] If $G$ is regular branch over~$\gamma_3(G)$ but
    not over $[G,G]$, then $\psi(\gamma_3(G))$ is subdirect in
  $G\times\cdots\times G$.
  \end{enumerate}
\end{proposition}

\begin{proposition}\label{prop:st-in-[G,G]}\cite[Prop.~4.2]{ThUr21}
  Let $G = G_\mathbf{E}$ be a multi-EGS group that is regular branch
  over $[G,G]$, and let $\dot{\mathbf{E}}$ denote the concatenation of
  the relevant systems $\mathbf{E}^{(1)}, \dots, \mathbf{E}^{(p)}$.
 If the $r_G$ vectors in $\dot{\mathbf{E}}$ are
  linearly independent, then we have $\St_G(r_G+1)\subseteq
  [G,G]$.
\end{proposition}

\begin{proposition}\cite[Prop.~3.11]{ThUr21} \label{pro:super-strongly-fractal}
  Every regular branch multi-EGS group is super strongly fractal.
\end{proposition}
  
Let $\mathcal{E}$ denote the subclass of $3$-generator multi-EGS
groups $\langle a,b^{(j)}, b^{(k)} \rangle$, where
$j, k\in \{1,\dots,p\}$ with $j<k$ and, subject to replacing the
generators $b^{(j)}, b^{(k)}$ with suitable powers, the associated
symmetric defining vectors
$\mathbf{e}_i^{(j)} = \mathbf{e} = (e_1,\dots,e_{p-1})$ and
$\mathbf{e}_i^{(k)} = \mathbf{f} = (f_1,\dots,f_{p-1})$ satisfy the
following condition: $e_i, f_i\in \{0,1\}$ and $e_i\ne f_i$ for all
$i\in\{1,\dots, p-1\}$. Next, we recall \cite[Thm.~1.1]{ThUr21} with
a small correction, which is to be justified
  in~\cite{ThUr21cor}.  The necessity of such a correction became
  apparent in connection with Theorem~\ref{thm:sharper-CSP} of the
  present paper.

\begin{proposition}[{\cite[Thm.~1.1]{ThUr21}, \cite{ThUr21cor}}] \label{prop:EGS-CSP} Let
  $G = G_\mathbf{E}$ be a multi-EGS group, and let $\dot{\mathbf{E}}$
  denote the concatenation of the relevant systems
  $\mathbf{E}^{(1)}, \dots, \mathbf{E}^{(p)}$.  Then $G$ has the
  congruence subgroup property if and only if $G \notin\mathcal{E}$
  and one of the following holds:
  \begin{enumerate}
  \item[\textup{(i)}] $G$ is regular branch over $[G,G]$ and the $r_G$
    vectors in $\dot{\mathbf{E}}$ are linearly independent;
  \item[\textup{(ii)}] $G$ is regular branch over~$\gamma_3(G)$ but
    not over $[G,G]$, and $r_G=2$.
  \end{enumerate}
\end{proposition}

Also, we document a result concerning multi-EGS groups
  that are regular branch over $\gamma_3(G)$ but not over $[G,G]$,
  which plays a key role in correcting the omission
  in~\cite[Thm.~1.1]{ThUr21}.

\begin{proposition} \cite{ThUr21cor} \label{prop:appB} 
    For $2\le r \le p$, let $1\le j_1<\dots<j_r\le p$ and let
    $\mathbf{e} \in \mathbb{F}_p^{\, p-1}$ be a non-constant,
    symmetric vector.  Consider the multi-EGS group
    \[ G = G_\mathbf{E} = \langle a,b_1,\dots,b_r \rangle \qquad
      \text{ associated to
        $\mathbf{E} = (\mathbf{E}^{(1)}, \dots,
        \mathbf{E}^{(p)})$,} \] where $r_{j_i }=1$,
    $\mathbf{E}^{(j_i)} = (\mathbf{e})$ and $b_i = b_1^{(j_i)}$
    denotes the single directed automorphism for $1 \le i \le r$ while
    $r_j = 0$ for $j \not\in \{ j_1, \dots, j_r \}$.  Furthermore,
    set $B = D \gamma_3(G)$, where
    \[
      D = \langle \Big\{ \prod\nolimits_{i=2}^r \big( b_{j_1}^{\, -1}
      b_{j_i}^{\, a^{j_i-j_1}} \big)^{\alpha_i}\mid \alpha_2,\dots,
      \alpha_r\in\mathbb{F}_p \textup{ such that }
      \sum\nolimits_{i=2}^r (j_i-j_1) \alpha_i= 0 \textup{ in }
      \mathbb{F}_p \Big\} \rangle^G \trianglelefteq G.
    \]

    Then $G$ is regular branch over $B$, and
    $\St_G(5)\subseteq B \subseteq \gamma_3(G)\St_G(n)$ for all
    $n\in\mathbb{N}$.
\end{proposition}

Next we complement Definition~\ref{def:standard-gens}, by
  discussing the invariant $\dot{r}_G$ of the multi-EGS group~$G$.
  
  \begin{definition}\label{def:rdot}
    We denote by $\dot{r}_G$ the maximal number of linearly
    independent defining vectors for the directed automorphisms in a
    standard generating system for~$G$, i.e.\ the dimension of
    the $\mathbb{F}_p$-subspace generated by the vector system
    $\dot{\mathbf{E}}$ resulting from concatenating the vector systems
    $\mathbf{E}^{(1)}, \dots, \mathbf{E}^{(p)}$
    that define $G = G_\mathbf{E}$.
  \end{definition}

Using ideas similar to \cite[Proof of Lem.~3.1]{Pe07} or
\cite[Proof of Lem.~4.1]{ThUr21}, one obtains the following result,
which is to be explained in more detail in~\cite{FeGuTh26}.

\begin{proposition}\cite[Cor.~4.2]{FeGuTh26}\label{prop:multi-GGS}
  Let $G = G_\mathbf{E}$ be a multi-EGS group that is regular branch
  over $[G,G]$.  Let $\dot{\mathbf{E}}$ denote the
    concatenation of the relevant systems
    $\mathbf{E}^{(1)}, \dots, \mathbf{E}^{(p)}$, and let $H$ be the
    multi-GGS group defined by the single vector system
    $\dot{\mathbf{E}}$.  Then
    $G \St_{\Aut T}(n) = H \St_{\Aut T}(n)$ for every
    $n \in \mathbb{N}$ and consequently the congruence completions
    $\overline{G}$ and $\overline{H}$ coincide.
\end{proposition}

\begin{remark} \label{rem:not-iso-but-prof-iso} From
  Proposition~\ref{prop:multi-GGS} and \cite{PeThXX} it is
  easy to construct pairs of multi-EGS groups, $G$ and $H$, such that
  $G$ and $H$ are profinitely isomorphic, but $G \not \cong H$.
  
  Indeed, it is enough to arrange for $G$ to be a multi-EGS group that
  is not a multi-GGS group and for $H$ to be a multi-GGS group, with
  $\dot{r}_G = r_G = r_H \ge 3$ and the concatenations of the defining
  vector systems for $G$ and $H$ being the same.  Under these
  conditions, Propositions~\ref{prop:EGS-facts} and \ref{prop:EGS-CSP}
  guarantee that $G$ and $H$ have the congruence subgroup property so
  that Proposition~\ref{prop:multi-GGS} implies
  $\widehat{G} \cong \overline{G} = \overline{H} \cong \widehat{H}$.
  As for $G \not \cong H$, we use \cite[Cor.~1]{GrWi03} and
  \cite[Proof of Cor.~3.8]{KlTh18} to reduce to $G$ and $H$ not being
  conjugate in $\Aut T$, which follows from~\cite[Prop.~3.5]{PeThXX}.
\end{remark}

From Propositions \ref{prop:multi-GGS} and \ref{prop:st-in-[G,G]},
together with \cite[Prop.~4.3]{AlKlTh16}, we deduce the following, which is of independent interest.

\begin{corollary}\label{prop:r-dot}
  Let $G = G_\mathbf{E}$ be a multi-EGS group that is
    regular branch over $[G,G]$, and let $n \ge \dot{r}_G+1$.  Then
    the minimal number of generators for the finite $p$-group
    $G / \St_G(n)$ is $1+\dot{r}_G$.
\end{corollary}

The corollary implies that, if the multi-EGS group $G$ is regular
branch over $[G,G]$, then the minimal number of generators for the
congruence completion $\overline{G}$ is $1+\dot{r}_G$. For
completeness, we remark that, if $G$ is regular branch over
$\gamma_3(G)$ but not over $[G,G]$, Proposition~\ref{prop:EGS-facts}
shows that $\dot{r}_G=1$. In this case, the minimal number of
generators for $\overline{G}$ follows a different pattern. Indeed, the
minimal number of generators for $\overline{G}$ is $2$ if $G$ is a
GGS-group, and it is $3$ otherwise as shown in the next
  proposition.  The latter and our proof are related
  to~\cite[Lem.~7.1]{FeGuTh26}.

  \begin{proposition}\label{pro:min-numb-gen-congr}
    Let $G$ be as in Proposition~\ref{prop:appB}. Then $3$ is the
    minimal number of generators for the congruence completion
    $\overline{G}$.
  \end{proposition}

  \begin{proof}
    First we show that $\overline{G} = \langle a, b_1, b_2 \rangle$ so
    that the minimal number of generators of the pro-$p$ group
    $\overline{G}$ is at most~$3$.  For this it is enough to prove
    that, for each $n \in \mathbb{N}$,
    \[
      b_k \in \langle a , b_1, b_2\rangle \, [G,G] \St_G(n) \qquad
      \text{for $3 \le k \le r$.}
    \]
    Let $n \in \mathbb{N}$ and suppose that $k \in \{3, \dots, r\}$.
    We put $\alpha = -(j_k-j_1)/(j_2-j_1) \in \mathbb{F}_p$ so that
    Proposition~\ref{prop:appB} implies
    \[
      b_1^{-(\alpha+1)} \bigl( b_2^{\, \alpha} \bigr)^{a^{j_2-j_1}}
      b_k^{\, a^{j_k - j_1}} \equiv_{[G,G]} \big( b_1^{\, -1}
      b_2^{\, a^{j_2-j_1}} \big)^\alpha \big( b_1^{\, -1}
      b_{j_k}^{\, a^{j_k-j_1}} \big) \in B \subseteq  [G,G] \St_G(n)
    \]
    which in turn yields $b_k \in \langle a , b_1, b_2\rangle [G,G] \St_G(n)$.
    
    To finish the proof, it suffices to show that $a, b_1, b_2$
    constitute a minimal generating set for $G$ modulo the particular
    congruence subgroup~$\St_G(3)$.  For this purpose we introduce
    some auxiliary notation.  The finite factor groups
    $G_2 = G/\St_G(2)$ and $G_3 = G/\St_G(3)$ act on corresponding
    truncated trees.  We denote the image of $g$ in $G_2$ by $\dot g$
    and its image in $G_3$ by $\ddot g$.  Observe that $\psi$ induces
    an embedding
    \[
      \overline{\psi} \colon \St_{G_3}(1)\longrightarrow G_2 \times
      \overset{p}\dots \times G_2
    \]
    satisfying
  $\overline{\psi}(\ddot b_i) = (\dot a^{e_{j_i}}, \dots, \dot
    a^{e_{p-1}}, \dot b_i, \dot a^{e_1}, \dots, \dot a^{e_{j_i -1}})$
    for $i\in\{1,\dots,r\}$.  Without loss of generality we may
    assume that~$j_1=1$.

    We show that $\ddot a, \ddot b_1, \ddot b_2$ constitute a minimal
    generating set for~$G_3$.  It is easily seen that
    $\dot a, \dot b_1$ form a minimal generating set of~$G_2$, and
    thus it suffices to show that
    $\ddot b_2 \not\in \langle \ddot a, \ddot b_1 \rangle$.  For a
    contradiction, we suppose otherwise and conclude that
    $\ddot b_2 \in \St_{G_3}(1) = \langle \ddot b_1, \ddot b_1^{\,
      \ddot a}, \dots, \ddot b_1^{\, \ddot a^{p-1}} \rangle$ which in
    turn yields
    \begin{equation}
      \label{equ:rel-betw-b1-and-b2} 
      \ddot b_2 \equiv \ddot b_1^{\,i_0} (\ddot b_1^{\, \ddot a}
      )^{i_1} \cdots (\ddot b_1^{\, \ddot a^{p-1}})^{i_{p-1}} 
      \qquad \mathrm{modulo} \quad
        [\St_{G_3}(1),\St_{G_3}(1)],
    \end{equation}
    for suitable $i_0, \dots, i_{p-1} \in \{0,\dots, p-1\}$.  We
    claim that $i_{p-j_2+1}=1$ and $i_k=0$ for all other indices $k$
    so that there is $\ddot c \in [\St_{G_3}(1),\St_{G_3}(1)]$ such that
    \begin{equation}
      \label{equ:strong-rel-betw-b1-and-b2} 
      \ddot b_2 = \ddot b_1^{\, \ddot a^{1-j_2}} \ddot c.
    \end{equation}   
    Recall that $\mathbf{e} = (e_1,\dots,e_{p-1})$ denotes the
    non-constant, symmetric vector underlying~$G = G_{\mathbf{E}}$.
    Using the group homomorphism
    \[
       \St_{G_3}(1)/[\St_{G_3}(1),\St_{G_3}(1)]\longrightarrow
      G_2/[G_2,G_2] \times \overset{p}\dots \times
      G_2/[G_2,G_2]
    \]
    which is induced by $\overline{\psi}$ and the identity
    $\dot b_1 \equiv_{[G_2,G_2]} \dot b_1^{\, \dot a^{1-j_2}} = \dot
    b_2$ in~$G_2$, we conclude from \eqref{equ:rel-betw-b1-and-b2}
    that
    \begin{multline*}
      (\dot a^{e_{j_2+1}}, \dots, \dot a^{e_{p-1}}, \dot b_1, \dot
      a^{e_1}, \dots, \dot a^{e_{j_2}}) %
      \equiv \overline{\psi}(\ddot b_2) %
      \equiv \overline{\psi}\bigl( \ddot b_1^{\,i_0} (\ddot b_1^{\,
        \ddot a})^{i_1} \cdots
      (\ddot b_1^{\, \ddot a^{p-1}})^{i_{p-1}}\bigr) \\
      \equiv (\dot b_1^{\, i_{1}} \dot a^{*}, \dots, \dot b_1^{\,
        i_{p-1}} \dot a^{*}, \dot b_1^{\, i_0} \dot a^{*}) \qquad
      \mathrm{modulo} \quad [G_2,G_2] \times \overset{p}{\dots} \times [G_2,G_2],
    \end{multline*}
    for unspecified exponents~$*$.  As
    $G_2 = \langle \dot a, \dot b_1 \rangle$ and
    $G_2/[G_2,G_2] \cong C_p \times C_p$, we arrive
    at~\eqref{equ:strong-rel-betw-b1-and-b2}.

    Next we make use of the group homomorphism
    \[
      \St_{G_3}(1)/\gamma_3(\St_{G_3}(1)) \longrightarrow
      G_2/\gamma_3(G_2) \times \overset{p}\dots \times
      G_2/\gamma_3(G_2)
    \]
    which is induced by $\overline{\psi}$.  From
    $\ddot c = (\ddot b_1^{\, \ddot a^{1-j_2}})^{-1} \ddot b_2$ we see
    that, modulo
    $\gamma_3(G_2) \times \overset{p}{\dots} \times \gamma_3(G_2)$,
    \[
      \overline{\psi}(\ddot c) \equiv (1, \dots,1,\dot b_1^{\, -1}
      \dot b_2,1,\overset{j_2-1}\dots,1) \equiv (1, \dots,1,[\dot
      b_1, \dot a^{1-j_2}], 1,\overset{j_2-1}\dots,1).
    \]
    Since $[\dot b_1, \dot a^{1-j_2}] \notin \gamma_3(G_2)$, the
    product of the coordinates of $\overline{\psi}(\ddot c)$ does not
    lie in~$\gamma_3(G_2)$.  We show below that this is incompatible
    with $\ddot c \in [\St_{G_3}(1),\St_{G_3}(1)]$ and thus arrive at
    the desired contradiction.

    Indeed, the coordinates of $\overline{\psi}(g)$, for any
    $g \in [\St_{G_3}(1),\St_{G_3}(1)]$, lie in $[G_2,G_2]$ and their
    product modulo $\gamma_3(G_2)$ is independent of the particular
    ordering.  Since
    \[
      [\St_{G_3}(1),\St_{G_3}(1)] = \langle [\ddot b_1^{\, \ddot a^k},
      \ddot b_1^{\, \ddot a^\ell}] \mid 1 \le k < \ell \le p \rangle \,
      \gamma_3(\St_{G_3}(1)),
    \]
    it is enough to verify that the product of the coordinates of
    $\overline{\psi}([\ddot b_1^{\, \ddot a^k}, \ddot b_1^{\, \ddot
      a^\ell}])$ is in $\gamma_3(G_2)$, for any choice of
    $1 \le k < \ell \le p$.  As $\mathbf{e}$ is symmetric and thus
    $e_{p-\ell+k} = e_{\ell-k}$, this follows from
    \begin{align*}
      \overline{\psi}([\ddot b_1^{\, \ddot a^k}, \ddot b_1^{\, \ddot a^\ell}]) %
      & = (1, \overset{k-1}{\dots}, 1, [\dot b_1, \dot
        a^{e_{p-\ell+k}}], 1, \overset{l-k-1}{\dots}, 1,
        [\dot a^{e_{\ell-k}}, \dot b_1], 1, \overset{p-\ell}{\dots}, 1)\\ 
      & = (1, \dots, 1, [\dot b_1,  \dot a]^{e_{p-\ell+k}}, 1, \dots,1,
        [\dot a, \dot b_1]^{e_{\ell-k}}, 1, \dots, 1). \qedhere
    \end{align*}
  \end{proof}

In light of Proposition~\ref{prop:multi-GGS}, we can at times replace
the consideration of certain multi-EGS groups with the simpler
multi-GGS groups.  However, when we consider general branch multi-EGS
groups, it is more convenient to simplify the notation similar to the
way in Proposition~\ref{prop:appB}; we formalise this below.

\medskip

\noindent\textit{Notation.}  When we investigate the normal subgroup
structure of multi-EGS groups $G = G_\mathbf{E}$ in the following
sections, it is unnecessary to distinguish carefully between the
different generating families $\mathbf{b}^{(j)}$,
$j \in \{1, \dots, p\}$.  It will be convenient to denote the
directed generators of $G$ by $b_1, \dots, b_r$ so that, for
instance, \eqref{equ:def-multi-EGS} simplifies to
$G = \langle a, b_1, \dots, b_r \rangle$.



\section{An effective version of the congruence subgroup
  property}\label{sec:sandwich}

In this section, we prove Theorem~\ref{thm:sharper-CSP}.  We use the
notation introduced in the previous sections.  In particular, the
multi-EGS groups $G$ that we consider are subgroups of the
automorphism group $\Aut T$ of the $p$-adic tree $T$, for some
prime~$p \ge 3$.

\begin{proposition}\label{prop:branching}
  Let $G$ be a multi-EGS group and let $N\trianglelefteq G$ be a
  normal congruence subgroup.  Let $m\in \mathbb{N}_0$ be maximal
  subject to $N\subseteq \St_G(m)$. Then the following hold:
  \begin{enumerate}
  \item [\textup{(i)}] if $G$ is regular branch over $[G,G]$ then
    \[
      \gamma_3(G)\times\overset{p^{m+1}}\dots\times \gamma_3(G)
      \subseteq \psi_{m+1}\big(\St_{[N,G]}(m+1)\big);
    \]
  \item [\textup{(ii)}] if $G$ is regular branch over $\gamma_3(G)$
    but not over $[G,G]$ then
    \[
      \gamma_4(G)\times\overset{p^{m+1}}\dots\times
      \gamma_4(G)\subseteq \psi_{m+1}\big(\St_{[N,G]}(m+1)\big).
    \]
  \end{enumerate}
\end{proposition}

\begin{proof}
  We write $r=r_G$ and $G = \langle a, b_1, \dots, b_r \rangle$ with
  directed generators $b_1,\dots,b_r$ as explained at the
    end of Section~\ref{sec:multi-EGS}, and we use the short notation
  $G' = [G,G]$.  We consider the cases (i) and~(ii) in parallel and
  accordingly write $K$ to denote either $G'$ or~$\gamma_3(G)$.
  Without further comment, we use the fact that
  $K \subseteq \St_G(1)$.  We put
  $S = \psi_{m+1}(\St_{[N,G]}(m+1)) \le G \times
  \overset{p^{m+1}}\dots \times G$.

  Since $G$ is spherically transitive and regular branch,
  Proposition~\ref{pro:super-strongly-fractal} shows that it suffices
  to establish that $S$ contains a system of \emph{normal} generators
  for $[K,G] \times 1 \times \overset{p^{m+1}-1} \dots \times 1$
  viewed as a subgroup of
  $G \times 1 \times \overset{p^{m+1}-1} \dots \times 1$.  Since $N$
  is a congruence subgroup, there is $\ell\in\mathbb{N}$ with
  $\ell \ge m+1$ such that $\St_G(\ell)\subseteq N$, and
    this gives
    \[
      \psi_{\ell-m-1}^{\, -1}(K \times \overset{p^{\ell-m-1}} \dots
      \times K) \times 1 \times \overset{p^{m+1}-1}\dots \times 1
      \subseteq \psi_{m+1}(\St_G(\ell)) \subseteq
      N.
    \]
    Working modulo
    $\psi_{\ell-m-1}^{\, -1}([K,G] \times \overset{p^{\ell-m-1}} \dots
    \times [K,G]) \times 1 \times \overset{p^{m+1}-1}\dots \times 1
    \subseteq \psi_{m+1}([\St_G(\ell),G]) \subseteq S$, we are
  effectively dealing with finite nilpotent images of the groups
  involved.  Thus it suffices to establish that $S$ contains a system
  of elements of the form $([x,y],1,\overset{p^{m+1}-1}\dots,1)$,
  where $x$ runs through $K$ and $y$ runs through a generating system
  for~$G$ modulo $\gamma_2(G) = G'$.

  Let $u$ denote the leftmost vertex at level~$m+1$.  As $G$ is
  regular branch over~$K$, we have
  $K \times 1 \times \overset{p^{m+1}-1} \dots \times 1 \subseteq
  \psi_{m+1}(\St_G(m+1))$.  Since $\St_{[N,G]}(m+1)$ is normal in $G$,
  it thus suffices to show:
  \begin{list}{}{ \setlength{\itemsep}{0pt} \setlength{\parsep}{1pt}
      \setlength{\labelwidth}{1cm} \setlength{\leftmargin}{1cm}
      \setlength{\itemindent}{0cm}}
  \item[({$\ast$})] For every $z \in G$ there is an element
    $\hat z \in \St_{[N,G]}(m+1) $ such that
    $\varphi_u(\hat z) \equiv z$ modulo $G'$.
  \end{list}
  
  From the choice of $m$ and the fact that $N$ is normal in the
  spherically transitive group~$G$, it follows that $N = \St_N(m)$ and
  that $\varphi_v(N)\not\subseteq \St_G(1)$ for every vertex
  $v$ at level~$m$.  Let $v$ denote the leftmost vertex at level~$m$
  and pick an element $h \in N$ with
  \[
    \varphi_v(h) = ac \qquad \text{for some $c \in \St_G(1)$.}
  \]
  By Proposition~\ref{pro:super-strongly-fractal}, there are elements
  $g_1, \dots, g_r \in \St_G(m)$ such that
  \[
    \varphi_v(g_i)=b_i \qquad \text{for $i \in \{1,\dots,r\}$.}
  \]
  Clearly, $H = \langle h,g_1, \dots, g_r \rangle \subseteq \St_G(m)$
  projects under $\varphi_v$ onto
  $L = \langle ac, b_1, \dots, b_r \rangle$.  Set
  $M = \langle [h,g_1], \dots, [h,g_r] \rangle^H \subseteq [N,G] \cap
  [H,H]$ and observe that
  $\varphi_v(M) = \langle [ac,b_1], \dots, [ac,b_r] \rangle^L
  \subseteq [L,L]$.

  To conclude the proof of ($\ast$), we observe that, modulo
  $\St_G(1)' = [\St_G(1),\St_G(1)] \trianglelefteq G$,
  \[
    [b_i,b_j] \equiv 1, \quad [ac,b_i] = [a,b_i]^c \,[c,b_i] \equiv
    [a,b_i] \quad \text{and} \quad [a,b_i]^{ac} \equiv [a,b_i]^a
    \qquad \text{for $i,j \in \{1,\dots,r\}$.}
  \]
  From this we deduce that $\varphi_v(M) \St_G(1)' = G' \St_G(1)'$.
  Recall that $\psi(G')$ is subdirect in
  $G \times \overset{p}\dots \times G$, by
  Proposition~\ref{prop:EGS-subdirect}, and observe that
  $\psi(\St_G(1)') \subseteq G' \times \overset{p}\dots \times G'$.
  We conclude that for every $z \in G$ there exists an element
  $\hat z \in M \subseteq \St_{[N,G]}(m+1)$ such that
  $\psi(\varphi_v(\hat z))$ takes the form
  $(z, *,\overset{p-1}\dots,*)$ modulo
  $G' \times \overset{p}\dots \times G'$, where $*$ functions as a
  placeholder for unspecified elements of~$G$.  In other words,
  $\hat z$ satisfies $\varphi_u(\hat z) \equiv z$ modulo $G'$.
\end{proof}

For the next result, we recall Definition~\ref{def:rdot} which
provides the invariant $\dot{r}_G$.

\begin{corollary}\label{cor:n-step-general}
  Let $G$ be a multi-EGS group, and let $N \trianglelefteq G$ be such
  that $[N,G]$ is a congruence subgroup.  Let $m\in \mathbb{N}_0$ be
  maximal subject to $N\subseteq \St_G(m)$.  Then the
  following hold:
 \begin{enumerate}
 \item [\textup{(i)}] if $G$ is regular branch over $[G,G]$ then
   $\St_G(m + \dot{r}_G + 3)\subseteq [N,G]$;
 \item [\textup{(ii)}] if $G$ is regular branch over $\gamma_3(G)$ but
   not over $[G,G]$ then $\St_G(m+7)\subseteq [N,G]$.
 \end{enumerate}
\end{corollary}

\begin{proof}
  (i) By Proposition~\ref{prop:branching}, it suffices to show that
  $\St_G(\dot{r}_G+2) \subseteq \gamma_3(G)$.  If $G$ has the
  congruence subgroup property, then Proposition~\ref{prop:EGS-CSP}
  yields $\dot{r}_G = r_G$ and with \cite[Prop.~3.9 and~4.2]{ThUr21}
  we obtain the desired inclusion.  Now suppose that $G$ does not have
  the congruence subgroup property.  Since $[N,G]$ is a congruence
  subgroup, there is $\ell \in \mathbb{N}$ such that
  $\St_G(\ell) \subseteq [N,G]$.  Working modulo $\St_G(\ell)$ and
  using Proposition~\ref{prop:multi-GGS}, we may suppose without loss
  of generality that $G$ is the corresponding multi-GGS group with the
  congruence subgroup property.  As before we obtain the desired
  inclusion.
  
  (ii) By Proposition~\ref{prop:branching}, it is enough to establish
  that $\St_G(6) \subseteq \gamma_4(G)$.  From
  \cite[Prop.~3.9]{ThUr21} we deduce that
  $\gamma_3(G)\times \overset{p}\dots\times \gamma_3(G)\subseteq
  \psi(\gamma_4(G))$, and since $[N,G]$ is a congruence subgroup, we
  conclude that it suffices to show:
  $\St_G(5) \subseteq \gamma_3(G) \St_G(\ell)$ for all
  $\ell \in \mathbb{N}$.  The latter holds by
  Proposition~\ref{prop:appB}.
\end{proof}

For GGS-groups $G$, the above results can be strengthened as follows,
where $G' = [G,G]$ and $G'' = [G',G']$ denote the first and the second
derived subgroups of $G$.

\begin{proposition}\label{prop:GGS}
  Let $G$ be a GGS-group and let $N\trianglelefteq G$ be a non-trivial
  normal subgroup.  Let $m\in \mathbb{N}_0$ be maximal subject to
  $N\subseteq \St_G(m)$.  Then the following hold:
 \begin{enumerate}
 \item [\textup{(i)}] if $G$ is regular branch over $G'$
   then
   \[
     G''\times\overset{p^m}\dots\times G'' \subseteq \psi_m([N,G])
   \]
   and consequently $\St_G(m+3)\subseteq [N,G]$;
 \item [\textup{(ii)}] if $G$ is regular branch over $\gamma_3(G)$ but
   not over $G'$, then
   \[
     \gamma_3(G)'\times\overset{p^m}\dots\times \gamma_3(G)'\subseteq
     \psi_m([N,G])
   \]
   and consequently $\St_G(m+4)\subseteq [N,G]$.
 \end{enumerate}
\end{proposition}

\begin{proof}
  Suppose that $G$ is regular branch over $K$, where $K$ is either
  $G'$ or $\gamma_3(G)$, depending on which of the two cases we are
  in.  Then $G \ne \mathcal{G}$ and, by
 Proposition~\ref{prop:EGS-subdirect}, the group $G$ has the
  congruence subgroup property and $\psi(G')$ is subdirect in
  $G\times \overset{p}\dots\times G$.  In particular,
  $\St_G(\ell) \subseteq N$ for sufficiently large $\ell \in \mathbb{N}$ and
  we proceed similar to the proof of Proposition~\ref{prop:branching}.

  \smallskip
  
  (i) Let $v$ denote the leftmost vertex at level~$m$.  As in the
  proof of Proposition~\ref{prop:branching}, there are elements
  $h \in \St_N(m) = N$ and $g\in \St_G(m)$ such that
  $\varphi_v(h) = ac$ with $c\in \St_G(1)$ and $\varphi_v(g)=b$.  This
  yields $[ac,b] = \varphi_v([h,g]) \in \varphi_v(\St_{[N,G]}(m))$.
  Since $G$ is super strongly fractal and since we are effectively
  working with finite nilpotent images of the groups involved, we
  conclude that
  $G' = \langle [a,b] \rangle^G \subseteq \varphi_v([N,G])$.  Recall
  that
  $G' \times 1 \times \overset{p^m-1}\dots \times 1 \subseteq
  \psi_m(\St_G(m))$, because $G$ is regular branch over~$G'$.  Forming
  commutators and using once more that $G$ is super strongly fractal,
  we conclude that
  $G'' \times 1 \times \overset{p^m-1}\dots \times 1 \subseteq
  \psi_m([N,G])$.  Since $G$ is spherically transitive and
  $[N,G] \trianglelefteq G$, this gives
  $G'' \times \overset{p^m}\dots \times G'' \subseteq \psi_m([N,G])$.
  The final statement follows from $\St_G(2) \subseteq \gamma_3(G)$
  and
  $\gamma_3(G)\times \overset{p}\dots\times \gamma_3(G)\subseteq
  \psi(G'')$; see Proposition~\ref{prop:GGS-facts} and, for the second
  inclusion, use that
  $G' \times \overset{p}\dots \times G' \subseteq \psi(G')$ and that
  $\psi(G')$ is subdirect in $G \times \overset{p}\dots \times G$.

  \smallskip

  (ii) We proceed as above to conclude that
  $\gamma_3(G)' \times \overset{p^{m}}\dots \times \gamma_3(G)'
  \subseteq \psi_m([N,G])$. From \cite[Proof of Thm.~2.7]{FeGaUr17} we
  see that
  $\gamma_4(G)\times \overset{p}\dots \times \gamma_4(G)\subseteq
  \psi(\gamma_3(G)')$ and
  $\gamma_3(G) \times \overset{p}\dots \times \gamma_3(G)\subseteq
  \psi(\gamma_4(G))$.  Together with $\St_G(2) \subseteq \gamma_3(G)$,
  these inclusions yield the final statement.
\end{proof}

To prove Theorem~\ref{thm:sharper-CSP} we need to strengthen our
results even further in the special case that $G$ is a
Fabrykowski--Gupta group.  We recall the following basic properties of
such groups.

\begin{lemma}\label{lem:FG-facts}
  Let $G$ be the Fabrykowski--Gupta group for the prime $p \ge 3$.
  \begin{itemize}
  \item[\textup{(a)}] For $m \ge 2$, the $m$th derived subgroup
    $G^{(m)}$ equals $\St_G(m)$ and
    \[
      \psi_{m-1}(\St_G(m)) = G' \times \overset{p^{m-1}}{\dots}
      \times G'.
    \]
  \item[\textup{(b)}] Let $m \in \mathbb{N}$ and $g \in \St_G(m)$.
    Then for each vertex $v$ at level $m-1$ there are integers
    $\ell(1), \dots, \ell(p) \in \{0,1,\dots,p-1\}$ such that
    \begin{equation}\label{equ:coordinate-link}
      \varphi_v(g) \equiv 
      \psi^{-1}\bigl(( a^{\ell(1)}
      b^{\ell(2)}, a^{\ell(2)} b^{\ell(3)}, \dots,
      a^{\ell(i)} b^{\ell(i+1)}, \dots, a^{\ell(p)}
      b^{\ell(1)} )\bigr) \qquad \mathrm{modulo} \quad \St_G(2).
    \end{equation}
  \end{itemize}
\end{lemma}

\noindent See \cite[Thm.~2.4(i) and~2.14, Lem.~3.4]{FeZu14} for part
(a) of Lemma~\ref{lem:FG-facts} and see \cite[Lem.~2.7]{Pe24} (also
\cite[2.2.2]{Si87}) for part (b).

\begin{proposition}\label{prop:2-jump}
  Let $G$ be the Fabrykowski--Gupta group for the prime $p \ge 3$.
  Let $N \trianglelefteq G$ be a non-trivial normal subgroup, and let
  $m \in \mathbb{N}$ be maximal subject to $N \subseteq
  \St_G(m)$. Then
  \[
    \gamma_3(G)\times \overset{p^m}{\dots} \times \gamma_3(G)
    \subseteq \psi_m ([N,G])
  \]
  and consequently $\St_G(m+2) \subseteq [N,G]$.
\end{proposition}

\begin{proof} 
  By Proposition~\ref{prop:EGS-CSP}, the group $G$ has the congruence
  subgroup property so that $[N,G]$ contains $\St_G(n)$ for some
  $n \in \mathbb{N}$ and $G/[N,G]$ is a finite $p$-group.  We observe
  that, for $m=0$, the result is straightforward: if
  $N \not\subseteq \St_G(1)$, then
  $\St_G(2) \subseteq \gamma_3(G) \subseteq G' \subseteq [N,G]$;
  compare with Proposition~\ref{prop:GGS-facts}.
  
  Now let $m \ge 1$.  Proposition~\ref{prop:GGS-facts} shows that
  \[
    \St_G(m+2) = \psi_m^{-1} \bigl( \St_G(2) \times
    \overset{p^m}{\dots} \times \St_G(2) \bigr) \subseteq \psi_m^{-1}
    \bigl( \gamma_3(G)\times \overset{p^m}{\dots} \times \gamma_3(G)
    \bigr).
  \]
  Hence the second assertion is a direct consequence of the first one.
  Furthermore, since $G$ is spherically transitive, it suffices to
  establish that
  \begin{equation}\label{equ:gamma3-included}
    \gamma_3(G)\times 1 \times \overset{p^m-1}{\dots} \times 1
    \subseteq \psi_m([N,G]),
  \end{equation}
  and because $G/[N,G]$ is nilpotent, there is no harm in working modulo
  \[
    \gamma_4(G)\times \overset{p^m}{\dots} \times \gamma_4(G)
    \trianglelefteq \psi_m(\St_G(m)).
  \]
  Below we shall locate elements $z_1, z_2 \in G$ satisfying
  $G = \langle z_1, z_2 \rangle G'$ and such that $\psi_m([N,G])$
  contains elements of the form
  \[
    ([z_1, z_2, z_1], 1, \overset{p^m-1}{\dots}, 1) \quad \text{and}
    \quad ([z_1, z_2, z_2], 1, \overset{p^m-1}{\dots}, 1) \qquad
    \mathrm{modulo} \quad \gamma_4(G)\times \overset{p^m}{\dots}
    \times \gamma_4(G).
  \]
  As
  $\gamma_3(G) = \langle [z_1,z_2,z_1], [z_1,z_2,z_2]
  \rangle\gamma_4(G)$, this suffices to
  deduce~\eqref{equ:gamma3-included}.

  \medskip
  
  Recall that $G = \langle a, b \rangle$, where $a$ denotes the rooted
  automorphism and the directed automorphism $b$ is given recursively
  by $\psi(b) = (a,1,\overset{p-2}{\dots},1,b)$.  Let
  $x \in N \smallsetminus \St_N(m+1)$; in particular,
  $x \in \St_G(m)$.  Since $G$ permutes the vertices at level $m$
  transitively, we may replace $x$ by $(x^g)^j$, for suitable
  $g \in G$ and $j \in \{0,1,\dots,p-1\}$, to ensure that $x$ takes
  the form
  \[
    \psi_m(x) = (a c_1, a^k c_2, *, \overset{p^m-2}{\dots} ,*) \qquad
    \text{with $c_1,c_2 \in \St_G(1)$ and $k \in \{0,1,\dots,p-1\}$,}
  \]
  where $*$ is used as a generic placeholder for elements of~$G$ whose
  specific nature is irrelevant for our argument.

  We proceed by case distinction, according to whether $c_1$, $c_2$
  belong to $G'$ or not.  The following notation is used across the
  cases: $v$ denotes the leftmost vertex at level~$m-1$ and, because
  $G$ is super strongly fractal
  (see Proposition~\ref{pro:super-strongly-fractal}), we can fix an
  element $\tilde a \in \St_G(m-1)$ such that
  $\varphi_v(\tilde a) = a^{-1}$.

  \medskip
  
  \noindent \textsl{Case 1:} $c_1, c_2 \in G'$.  From
  \eqref{equ:coordinate-link} we deduce that $k=0$, and hence
  \[
    \psi_m(x) \equiv (a,1, *, \overset{p^m-2}{\dots} ,*) \qquad
    \mathrm{modulo} \quad G' \times \overset{p^m}{\dots} \times G'.
  \]
  Since $G$ is super strongly fractal we can choose $g \in \St_G(m-1)$
  such that
  $$\varphi_v(g) = b^a = \psi^{-1}((b,a,1,\overset{p-2}{\dots},1)).
  $$
  Then
  \[
    \psi_m([x,g]) \equiv \bigl( [a,b], 1, \overset{p-1}{\dots}, 1 \,{,}\,
    *, \overset{p^m-p}{\dots}, * \bigr) \qquad \mathrm{modulo} \quad
    \gamma_3(G)\times \overset{p^m}{\dots} \times \gamma_3(G).
  \]
  Observe that
  \[
  \psi \bigl( [b,a]^{a^{-1}} \bigr) =
    (a,1,\overset{p-3}{\dots},1,b^{-1},a^{-1}b) \qquad \text{and}
    \qquad \psi \bigl( [a,b]^a \bigr) =
    (b,b^{-1}a,a^{-1},1,\overset{p-3}{\dots},1).
  \]
  From
  $G' \times 1 \times \overset{p^{m-1}-1}{\dots} \times 1 \subseteq
  \psi_{m-1}(\St_G(m-1))$, we conclude that there are elements
  $h_1,h_2 \in \St_G(m-1)$ such that
  $\psi_{m-1}(h_1) = ([b,a]^{a^{-1}},1,\overset{p^{m-1}-1}{\dots},1)$
  and $\psi_{m-1}(h_2) = ([a,b]^a,1,\overset{p^{m-1}-1}{\dots},1)$,
  thus $h_1,h_2 \in \St_G(m)$ and
  \[
    \psi_m (h_1) = (a,*,\overset{p-1}{\dots},* \,{,}\,
    1,\overset{p^m-p}{\dots},1) \qquad \text{and} \qquad \psi_m(h_2)
    = (b,*,\overset{p-1}{\dots},* \,{,}\,
    1,\overset{p^m-p}{\dots},1).
  \]
  Modulo
  $\gamma_4(G) \times \overset{p^m}{\dots} \times \gamma_4(G)$, this
  yields
  \begin{align*}
    ([a,b,a], 1,\overset{p^m-1}{\dots},1) %
    & \equiv \psi_m([x,g,h_1]) \in \psi_m([N,G]), \\
    ([a,b,b], 1,\overset{p^m-1}{\dots},1) %
    & \equiv \psi_m([x,g,h_2]) \in \psi_m([N,G]).
  \end{align*}   
  This yields \eqref{equ:gamma3-included}, as explained at the
  beginning of the proof.
  
  \medskip

  \noindent \textsl{Case 2:} $c_1 \not\in G'$ and $c_2 \in G'$.  From
  \eqref{equ:coordinate-link} we deduce that $k \ne 0$ and
  \[
    \psi_m(x) \equiv (a c_1,a^k, *, \overset{p^m-2}{\dots} ,*) \qquad
    \mathrm{modulo} \quad G' \times \overset{p^m}{\dots} \times G'.
  \]
  Pick $j \in \{1,2,\dots,p-1\}$ such that $jk \equiv_p -1$.  Recall
  that $\tilde a \in \St_G(m-1)$ is such that
  $\varphi_v(\tilde a) = a^{-1}$.  Then $y := x(x^{\tilde a})^j \in N$
  satisfies
  \[
    \psi_m(y) \equiv \bigl( c_1, *, \overset{p^m-1}{\dots}, * \bigr)
    \qquad \mathrm{modulo} \quad G' \times \overset{p^m}{\dots} \times
    G'.
  \]
  From
  $G' \times 1 \times \overset{p^m-1}{\dots} \times 1 \subseteq
  \psi_m(\St_G(m))$, we conclude that there exists
  $h \in \St_G(m)$ such that
  $\psi_m(h) = ([ac_1,c_1],1,\overset{p^m-1}{\dots},1)$.  Modulo
  $\gamma_4(G) \times 1 \times \overset{p^m-1}{\dots} \times 1$, this yields
  \begin{align*}
    ([ac_1,c_1,ac_1], 1,\overset{p^m-1}{\dots},1) %
    & \equiv \psi_m([h,x]) \in \psi_m([N,G]), \\
    ([ac_1,c_1,c_1], 1,\overset{p^m-1}{\dots},1) %
    & \equiv \psi_m([h,y]) \in \psi_m([N,G]).
  \end{align*}   
  Using $G= \langle ac_1, c_1 \rangle G'$, we deduce
  \eqref{equ:gamma3-included}, as before.

  \medskip

  \noindent \textsl{Case 3:} $c_1 \in G'$ and $c_2 \not\in G'$.  This
  case is rather similar to the previous one, and note that  from
  \eqref{equ:coordinate-link} we again have $k=0$.  Putting
  $y  :=  x^{\tilde a} \in N$, we see that
  \[
    \psi_m(x) \equiv (a,   c_2, *, \overset{p^m-2}{\dots} ,*) \quad
    \text{and} \quad \psi_m(y) \equiv \bigl( c_2, *,
    \overset{p^m-1}{\dots}, * \bigr) \qquad \mathrm{modulo} \quad G'
    \times \overset{p^m}{\dots} \times G'.
  \]
  From
  $G' \times 1 \times \overset{p^m-1}{\dots} \times 1 \subseteq
  \psi_m(\St_G(m))$, we conclude that there exists $h \in \St_G(m)$
  such that $\psi_m(h) = ([a,c_2],1,\overset{p^m-1}{\dots},1)$.
  Modulo
  $\gamma_4(G) \times 1 \times \overset{p^m-1}{\dots} \times 1$, this
  yields
  \begin{align*}
    ([a,c_2,a], 1,\overset{p^m-1}{\dots},1) 
    & \equiv \psi_m([h,x]) \in \psi_m([N,G]), \\
    ([a,c_2,c_2], 1,\overset{p^m-1}{\dots},1) %
    & \equiv \psi_m([h,y]) \in \psi_m([N,G]).
  \end{align*}   
  Using $G= \langle a, c_2 \rangle G'$, we deduce
  \eqref{equ:gamma3-included}, as before.
  
  \medskip

  \noindent \textsl{Case 4:} $c_1, c_2 \not\in G'$.  In this
  situation, \eqref{equ:coordinate-link} yields
  \begin{align*}
    \psi(\varphi_v(x)) %
    &\equiv \bigl( a^{\ell(1)} b^{\ell(2)}, a^{\ell(2)}
      b^{\ell(3)},a^{\ell(3)} b^{\ell(4)},
      \dots, a^{\ell(p-1)} b^{\ell(p)}, a^{\ell(p)} b^{\ell(1)} \bigr) \\
    &\equiv \bigl( a b^k, a^k b^{\ell(3)},a^{\ell(3)} b^{\ell(4)}, \dots,
      a^{\ell(p-1)} b^{\ell(p)}, a^{\ell(p)} b \bigr) \quad \mathrm{modulo} \quad
      G' \times \overset{p}{\dots} \times G',
  \end{align*}
  where $\ell(1)=1$, $\ell(2)=k \ne 0$ and
  $\ell(3), \dots, \ell(p) \in \{0,1,\dots,p-1\}$ are determined by
  $x$. Recall that $\tilde a \in \St_G(m-1)$ is such that
  $\varphi_v(\tilde a) = a^{-1}$.  Considering elements of the form
  $(x^{-j} x^{\tilde a})^{{\tilde a}^i}$ for
  $i ,j\in\{0,1,\dots, p-1\}$, we see that we may return to Case~1 or
  Case~3 (i.e. in one component the total $b$-exponent is zero but the
  total $a$-exponent is non-zero), unless
  $\ell(i) \equiv_p k \cdot \ell(i-1)$ for $i \in\{2,\dots ,p\}$ and
  furthermore $1 =\ell(1) \equiv_p k \cdot \ell(p) = k^p \equiv_p k$.
  The latter implies $k= \ell(1) = \dots = \ell(p) =1$ so that we are
  reduced to the situation
  \[
    \psi_m(x) \equiv \bigl( ab, \overset{p}{\dots}, ab \,{,}\, *,
    \overset{p^m-p}{\dots},* \bigr) \qquad \mathrm{modulo} \quad G'
    \times \overset{p^m}{\dots} \times G'.
  \]
  Furthermore, from considering all other vertices at level $m-1$, we
  may also assume that, modulo
  $G' \times \overset{p^m}{\dots} \times G'$,
  \begin{equation}\label{eq:p-blocks}
    \psi_m(x) \equiv \bigl( ab, \overset{p}{\dots}, ab \;{,}\;
    (ab)^{k_2}, \overset{p}{\dots}, (ab)^{k_2} \;{,}\; \dots \;{,}\;
    (ab)^{k_{p^{m-1}}}, \overset{p}{\dots}, (ab)^{k_{p^{m-1}}} \bigr) 
  \end{equation}
  for some $k_2,\dots, k_{p^{m-1}}\in\{0,1,\dots,p-1\}$.
  
  For $n \in\{1,\dots, m-1\}$, the inclusion
  $G' \times 1 \times \overset{p^n-1}{\dots} \times 1 \subseteq
  \psi_n(\St_G(n))$ allows us to pick $g_n \in \St_G(n)$ satisfying
  \[
    \psi_n(g_n) = ([b,a],1,\overset{p^n-1}{\dots},1).
  \]
  We set $g = b^a g_1 \cdots g_{m-1}$ and observe from
  \[
   \psi(b^a g_1) = (b^a,a,1,\overset{p-2}{\dots},1) \qquad
    \text{and} \qquad \psi_{n-1}(g_n) =
    (g_1,1,\overset{p^{n-1}-1}{\dots},1) \quad \text{for
      $n \in\{2,\dots, m-1\}$}
  \]
  that $g$ fixes the vertex $v$ and
  \[
    \psi(\varphi_v(g)) = \psi(b^a) = (b,a,1,\overset{p-2}{\dots},1).
  \]

  For $m=1$, we have $g = b^a$ and the congruence \eqref{eq:p-blocks}
  yields
  \[
    \psi([x,g]) = ([a,b],[b,a],1,\overset{p-2}\dots,1) \qquad
    \mathrm{modulo} \quad \gamma_3(G) \times \overset{p}{\dots}
    \times \gamma_3(G).
  \]
  Moreover $h_1=b$ and $h_2=b^a$ satisfy $\psi(h_1) = (a,1,\dots,1,b)$
  and $\psi(h_2) = (b,a,1,\dots,1)$.  Modulo
  $\gamma_4(G)\times \overset{p}{\dots} \times \gamma_4(G)$, we
  obtain
  \begin{align*}
    (    [a,b,a], 1,\overset{p-1}{\dots},1) %
    & \equiv \psi([ [x,g],h_1]) \in \psi([N,G]), \\
    ([a,b,b], 1,\overset{p-1}{\dots},1) %
    & \equiv \psi_m([ [g,x]^{a^{-1}},h_2]) \in \psi([N,G]).
  \end{align*}   
  This yields \eqref{equ:gamma3-included}, as before.

  Now suppose that $m \ge 2$.  Observe that at every positive level,
  the element $g$ has exactly one non-trivial label being~$a$.
  Furthermore, $g$ fixes the vertex $u$ just above $v$, viz.\ the
  leftmost vertex at level $m-2$, has label $1$ at $u$ and satisfies
  $\psi(\varphi_u(g)) = (b^a,a,1,\overset{p-2}\dots,1)$.  Taking into
  account this information about $g$ together with the form of $x$
  indicated in~\eqref{eq:p-blocks}, we obtain
  \[
    \psi_m([x,g]) \equiv
    ([a,b],[b,a],1,
    \overset{p^2-2}{\dots},1\,{,}\, *,\overset{p^m-p^2}\dots,*) \quad
    \text{mod
      $\gamma_3(G) \times \overset{p^2}{\dots} \times \gamma_3(G)
      \times G' \times \overset{p^m-p^2}{\dots} \times G'$.}
  \]
  As seen in Case~1 there are elements $h_1,h_2 \in \St_G(m)$ such
  that
  \[
    \psi_m (h_1) = (a,1,\overset{p-3}{\dots},1,*,* \;{,}\;
    1,\overset{p^m-p}{\dots},1) \qquad \text{and} \qquad \psi_m(h_2)
    = (b,*,*,1,\overset{p-3}{\dots},1 \;{,}\;
    1,\overset{p^m-p}{\dots},1).
  \]
  Recall that $\tilde a \in \St_G(m-1)$ is such that
  $\varphi_v(\tilde a) = a^{-1}$.  Modulo
  $\gamma_4(G)\times \overset{p^m}{\dots} \times \gamma_4(G)$, we
  obtain
  \begin{align*}
    (    [a,b,a], 1,\overset{p^m-1}{\dots},1) %
    & \equiv \psi_m([ [g,x]^{\tilde a},h_1]) \in \psi_m([N,G]), \\
    ([a,b,b], 1,\overset{p^m-1}{\dots},1) %
    & \equiv \psi_m([ [g,x]^{\tilde a},h_2]) \in \psi_m([N,G]).
  \end{align*}   
  This yields \eqref{equ:gamma3-included}, as before.
\end{proof}

\begin{remark}\label{rmk:CSP-bound} We observe that Propositions~\ref{prop:branching}, \ref{prop:GGS}
  and~\ref{prop:2-jump} improve the following fact from \cite[Proof of Thm.~4]{NewHorizons}: for $G$ regular branch over a subgroup $K$,  if $N \trianglelefteq G$ is such that  $N \subseteq
  \St_G(m)$ with  $m \in \mathbb{N}_0$ maximally chosen, then $\psi_{m+1}^{-1}(K'\times \overset{p^{m+1}}{\dots}\times K')\subseteq [N,G]$. Hence if $k\in \mathbb{N}_0$ is such that $\St_G(k)\subseteq K'$, then $\St_G(m+k+1)\subseteq [N,G]$.
  \end{remark}

\begin{proof}[Proof of Theorem~\ref{thm:sharper-CSP}]
  The theorem simply summarises the results from
  Corollary~\ref{cor:n-step-general} and Propositions~\ref{prop:GGS}
  and~\ref{prop:2-jump}.
\end{proof}


\section{Twisted direct sums and normal subgroups}\label{sec:chain}

As before, let $T$ be the automorphism group of the $p$-adic tree, and
let $a \in \Aut T$ denote the rooted $p$-cycle permuting transitively
the first level vertices. Let $S = S_p$ be the Sylow pro-$p$ subgroup
of $\Aut T$ consisting of all elements whose labels are powers
of~$a$. Throughout this section let
\[
  G = \langle a \rangle \ltimes \St_G(1) \le S
\]
be a self-similar subgroup containing~$a$, and write
$\psi \colon \St_G(1) \to G \times \overset{p}{\dots} \times G$ for
the standard `geometric' embedding, which is
$\langle a \rangle$-equivariant. Let $V$ be a finite
$\mathbb{F}_pG$-module. The \emph{$\psi$-twisted $p$-fold direct sum
  of $V$}, denoted by $V \vert^{\oplus p}_\psi$, is the finite
$\mathbb{F}_pG$-module defined as follows. The underlying vector space
of $V \vert^{\oplus p}_\psi$ is the $p$-fold direct sum
$V^{\oplus p} = V \oplus \overset{p}{\dots} \oplus V$, the element
$a$ acts on $V^{\oplus p}$ by cyclic permutation of the $p$ summands,
and $\St_G(1)$ acts on $V^{\oplus p}$ via $\psi$ and in accordance
with the natural action of $G \times \overset{p}{\dots} \times G$ on
$V^{\oplus p}$.

Using this construction we show that, for every $m \in \mathbb{N}$,
the $G$-action on the elementary abelian groups $\St_S(m)/\St_S(m+1)$
is uniserial; this means that for every non-trivial $G$-invariant
subgroup $H$ of $\St_S(m)/\St_S(m+1)$ the index of $[H,G]$ in $H$ is
equal to~$p$. Consequently, the $\mathbb{F}_pG$-submodules of each
section $\St_S(m)/\St_S(m+1)$ form a chain. In this way we get a
direct handle on the normal subgroups $N \trianglelefteq_\mathrm{o} G$
satisfying $\St_G(m+1)\subsetneq N\subseteq \St_G(m)$ for some
$m\in\mathbb{N}$.

\begin{definition}
  Let $W_0 = \mathbb{F}_p$ be the trivial $\mathbb{F}_pG$-module
  and, for $m \in \mathbb{N}$, we define recursively
  \[
    W_m = W_{m-1} \vert^{\oplus p}_\psi.
  \]
  Furthermore, we observe that there is a natural isomorphism of
  $\mathbb{F}_pG$-modules
  \begin{equation} \label{equ:iso-Wm}
    \St_S(m)/\St_S(m+1) \cong W_m \quad \text{for
      $m \in \mathbb{N}_0$.}
  \end{equation}
\end{definition}

Next we describe, for each $m \in \mathbb{N}$, a family of natural
$\mathbb{F}_pG$-submodules $V_\mathbf{j}$ of $W_m$, indexed by
elements $\mathbf{j} = (j_1,\dots, j_m)$ of the parameter set
$J_m = \{1,2,\dots,p\}^m$; subsequently, our aim will be to prove
that the modules $V_\mathbf{j}$ provide all non-trivial submodules,
subject to suitable extra conditions on~$G$.  The trivial submodule is labelled by an additional parameter $(0,p, \dots, p)$.

\begin{definition} \label{def:W1-and-submodules}
  Clearly, the $\mathbb{F}_pG$-module $ W_1$ admits the  $\mathbb{F}_pG$-submodules
  \[
    V_{(j)} =W_1.(a-1)^{p-j} \quad \text{with
      $\dim_{\mathbb{F}_p}(V_{(j)}) = j$,} \qquad \text{for $0 \le j \le p$.}
  \]
  Incidentally, it is not difficult to describe explicit $\mathbb{F}_p$-bases for
  these submodules:
  \begin{align*}
    V_{(p)} & = \mathrm{span}_{\mathbb{F}_p} \{ (1,0,\dots,0),
              (0,1,0,\dots,0), \dots, (0,\dots,0,1) \}, \\
    V_{(p-1)} & = \mathrm{span}_{\mathbb{F}_p} \{
                (1,-1,0,\dots,0),  (0,1,-1,0,\dots,0), \dots,
                (0,\dots,0,-1,1) \}, \\
    V_{(p-2)} & = \mathrm{span}_{\mathbb{F}_p} \{ (1,-2,1,0,\dots,0),
                (0,1,-2,1,0,\dots,0), \dots, (0,\dots,0,1,-2,1) \},\\
            &\,\,\,\,\vdots\\
    V_{(1)} & = \mathrm{span}_{\mathbb{F}_p} \{ (1,\dots,1)\}, \qquad
              \text{and} \qquad V_{(0)} = \{0\}.
  \end{align*}

  Now let $m \in \mathbb{N}$.  For any
  $\mathbf{j} = (j_1,\dots,j_m) \in \mathbb{N}_0^{\, m}$ we write
  \[
    \mathbf{j}' = (j_1, \dots, j_{m-1}) \qquad \text{and} \qquad
    \mathbf{j} = \mathbf{j}' \boxplus (j_m).
  \]
  The \emph{predecessor} of $\mathbf{j} = (j_1,\dots,j_m) \in J_m$ is 
  defined to be
  \[
    \mathbf{j}^- =
    \begin{cases}
      \mathbf{j}' \boxplus (j_m-1) & \text{for $2 \le j_m \le p$,}  \\
      (\mathbf{j}')^- \boxplus (p) & \text{for $j_m = 1$ and $m >
        1$,} \\
      (0) & \text{for $\mathbf{j} = (1)$.}
    \end{cases}
  \]
  We observe that $\mathbf{j}^- \in J_m$ unless
  $\mathbf{j} = (1,\dots,1)$, in which case
  $\mathbf{j}^- = (0,p,\dots,p)$.  In fact, the predecessor relation
  induces a linear order on the augmented parameter set
  $J_m \cup \{(0,p,\dots,p)\}$, viz.\ the lexicographic order.
  
  For $m=1$ we already identified $V_\mathbf{j}$ as a submodule of the
  $\mathbb{F}_pG$-module $W_1$, and we observe that
  $V_\mathbf{j} / V_{\mathbf{j}^-} \cong W_0$ for
  $\mathbf{j} \in J_1$.  Now suppose that $m \ge 2$ and consider
  $\mathbf{j} = (j_1,\dots,j_m) \in J_m$.  By
  recursion, $V_{\mathbf{j}'}$ and $V_{(\mathbf{j}')^-}$ are
  submodules of $W_{m-1}$, and  $V_{\mathbf{j}'}/V_{(\mathbf{j}')^-} \cong W_0$.  Thus
  $V_{\mathbf{j}'} \vert^{\oplus p}_\psi$ and
  $V_{(\mathbf{j}')^-} \vert^{\oplus p}_\psi$ are naturally submodules
  of the $\mathbb{F}_pG$-module $W_{m-1} \vert^{\oplus p}_\psi = W_m$
  and
  \[
    \bigl( V_{\mathbf{j}'} \vert^{\oplus p}_\psi \bigr) / \bigl( V_{(\mathbf{j}')^-}
    \vert^{\oplus p}_\psi \bigr) \cong W_1.
  \]
  We define $V_\mathbf{j}$ to be the submodule of $W_m$ that lies
  between $V_{\mathbf{j}'} \vert^{\oplus p}_\psi$ and
  $V_{(\mathbf{j}')^-} \vert^{\oplus p}_\psi$ and corresponds to the
  submodule $V_{(j_m)}$ of
  $W_1 \cong \bigl( V_{\mathbf{j}'} \vert^{\oplus p}_\psi \bigr) /
  \bigl( V_{(\mathbf{j}')^-} \vert^{\oplus p}_\psi \bigr)$.  In
  addition, we set $V_{(0,p,\dots,p)} = \{0\}$, the trivial
  submodule of~$W_m$.
\end{definition}

An elementary, but useful feature of the above construction is that,
for $m \ge 1$ and $\mathbf{j} \in J_m$, the module $V_\mathbf{j}$ lies
subdirectly inside the $\mathbb{F}_pG$-module
$V_{\mathbf{j}' \boxplus (p)} = V_{\mathbf{j}'} \vert^{\oplus
  p}_\psi$.  (For $m=1$, we pragmatically agree to read $V_{()}$ as
$W_0 = \mathbb{F}_p$.)

\begin{example}
  The $\mathbb{F}_pG$-modules just defined yield a descending chain of
  submodules
  \begin{multline*}
    W_3 = \underline{V_{(p,p,p)}} \supsetneq V_{(p,p,p-1)}
    \supsetneq \cdots \supsetneq V_{(p,p,1)} \supsetneq
    \underline{V_{(p,p-1,p)}} \supsetneq V_{(p,p-1,p-1)} \supsetneq
    \cdots \supsetneq  V_{(p,p-1,1)} \\
    \supsetneq \underline{V_{(p,p-2,p)}} \supsetneq \cdots \,\,\cdots
    \supsetneq V_{(1,2,1)} \supsetneq \underline{V_{(1,1,p)}}
    \supsetneq V_{(1,1,p-1)} \supsetneq \cdots \supsetneq
    V_{(1,1,1)} \supsetneq \underline{V_{(0,p,p)}} = \{0\},
    \end{multline*}
    with each term of index $p$ inside its predecessor.  The
    underlined terms $V_{(i,j,p)} = V_{(i,j)} \vert^{\oplus p}_\psi$
    are the ones that arise naturally from the terms of the
    corresponding filtration for~$W_2$, by recursion.
\end{example}

\begin{proposition} \label{pro:submodules-of-Wm}
  Suppose that $G$ contains a directed automorphism $b \in \St_S(1)$
  such that
  \[
    \psi(b) = (a^{e_1},\dots,a^{e_{p-1}},b) \quad \text{with}
    \quad \sum\nolimits_{i=1}^{p-1} e_i \not\equiv_p 0,
  \]
  and let $m \in \mathbb{N}$.  Then the modules $V_\mathbf{j}$,
  $\mathbf{j} \in J_m$, are precisely the non-trivial submodules of
  the $\mathbb{F}_pG$-module $W_m$.

  They form a descending chain, with $V_{\mathbf{j}^-}$ being the
  unique maximal submodule of $V_\mathbf{j}$ and
  $V_\mathbf{j}/V_{\mathbf{j}^-} \cong W_0 = \mathbb{F}_p$ for each
  $\mathbf{j} \in J_m$.  In particular, every
  $\mathbb{F}_pG$-submodule of $W_m$ is cyclic.
\end{proposition}

\begin{proof}
  There is no harm in assuming that $G = \langle a,b \rangle$ is the
  non-torsion GGS-group generated by $a$ and~$b$.  We argue by
  induction on~$m$.  For $m=1$, the action of $G$ on $W_1$ factors
  through $G / \St_G(1) = \langle \overline{a} \rangle \cong C_p$ and
  the situation can be explicitly described as in
  Definition~\ref{def:W1-and-submodules}.

  Now suppose that $m \ge 2$.  Clearly, the modules $V_\mathbf{j}$,
  $\mathbf{j} \in J_m$, form a descending chain, with
  $V_{\mathbf{j}^-}$ a maximal submodule of $V_\mathbf{j}$ and
  $V_\mathbf{j}/V_{\mathbf{j}^-} \cong \mathbb{F}_p$ for each
  $\mathbf{j} \in J_m$.  Thus it suffices to show that, for
  $\mathbf{j} = (j_1,\dots,j_m) \in J_m \smallsetminus
  \{(1,\dots,1)\}$ and
  $v \in V_\mathbf{j} \smallsetminus V_{\mathbf{j}^-}$,
  \begin{enumerate}
  \item
    $v (a-1) \in V_{\mathbf{j}^-} \smallsetminus V_{\mathbf{j}^{--}}$
    and $v (b-1) \in V_{\mathbf{j}^{--}}$ if $j_m \ne 1$;
  \item $v (a-1) \in V_{\mathbf{j}^{--}} $ and
    $v (b-1) \in V_{\mathbf{j}^-} \smallsetminus V_{\mathbf{j}^{--}}$
    if $j_m = 1$.
  \end{enumerate}
  Assertion (i) follows directly from the definitions, in particular
  $V_{\mathbf{j}' \boxplus (p) } \supseteq V_\mathbf{j} \supseteq
  V_{(\mathbf{j}')^- \boxplus (p)}$, and and the fact that the module
  \[
    V_{\mathbf{j}' \boxplus (p) } / V_{(\mathbf{j}')^- \boxplus (p)} =
    \bigl( V_{\mathbf{j}'} \vert^{\oplus p}_\psi \bigr) / \bigl(
    V_{(\mathbf{j}')^-} \vert^{\oplus p}_\psi \bigr) \cong W_1
  \]
  is fully understood.

  It remains to establish~(ii).  Suppose that $j_m=1$ and write
  $\mathbf{i} = \mathbf{j}'$ for short.  Modulo
  $V_{\mathbf{j}^-} = V_{\mathbf{i}^- \boxplus (p)} = V_{\mathbf{i}^-}
  \vert^{\oplus p}_\psi$, we may write
  $v \in V_\mathbf{j} = V_{\mathbf{i} \boxplus (1)} \le V_{\mathbf{i}}
  \vert^{\oplus p}_\psi$ as $(v_1,\dots,v_1)$ with
  $v_1 \in V_\mathbf{i} \smallsetminus V_{\mathbf{i}^-}$.  From
  $V_{\mathbf{j}^-} (a-1) \subseteq V_{\mathbf{j}^{--}}$ and
  $(v_1,\dots,v_1) (a-1) = (0,\dots,0)$ we deduce that
  $v(a-1) \in V_{\mathbf{j}^{--}}$.  Similarly,
  $V_{\mathbf{j}^-} (b-1) \subseteq V_{\mathbf{j}^{--}}$ and we
  obtain
  \[
    v(b-1) \equiv \big( v_1 (a^{e_1}-1), \dots, v_1 (a^{e_{p-1}}-1),
    v_1 (b-1) \big) \qquad \mathrm{modulo} \quad
      V_{\mathbf{j}^{--}}.
  \]
  In order to show that
  $v(b-1) \not\in V_{\mathbf{j}^{--}} = V_{\mathbf{i}^- \boxplus
    (p-1)}$ we need to establish that
  \begin{equation} \label{equ:quersumme} \sum\nolimits_{i=1}^{p-1} v_1
    (a^{e_i}-1) \;+\; v_1(b-1) \not\equiv 0 \qquad \mathrm{modulo}
    \quad V_{\mathbf{i}^{--}}.
  \end{equation}
  By induction, applied to
  $v_1 \in V_\mathbf{i} \smallsetminus V_{\mathbf{i}^-}$, there are
  two cases: either (i)' $v_1(a-1) \not\equiv 0$ and $v_1(b-1) \equiv
  0$, or
  (ii)' $v_1(a-1) \equiv 0$ and $v_1(b-1) \not\equiv 0$ modulo
  $V_{\mathbf{i}^{--}}$.

  In case (i)' we deduce from $v_1 (a-1) \not\equiv 0$ and
  $v_1 (a-1)^2 \equiv 0$ modulo $V_{\mathbf{i}^{--}}$ that
  \begin{align*}
    \sum\nolimits_{i=1}^{p-1} v_1 (a^{e_i}-1) &= v_1 (a-1)
    \sum\nolimits_{i=1}^{p-1}
    (a^{e_i-1} + a^{e_i-2}+\dots + a + 1) \\
   &\equiv \underbrace{v_1 (a-1)}_{\not\equiv 0}
    \underbrace{\sum\nolimits_{i=1}^{p-1} e_i}_{\not\equiv_p 0} \not
    \equiv 0 \qquad \mathrm{modulo} \quad V_{\mathbf{i}^{--}},
  \end{align*}
  thus $v_1(b-1) \equiv 0$ modulo $V_{\mathbf{i}^{--}}$ implies
  \eqref{equ:quersumme}.  In case (ii)' we obtain directly
  \eqref{equ:quersumme}.
\end{proof}

Next we aim to describe the normal subgroups $N$ of the self-similar group $G$ that lie between two consecutive terms of the filtration
$\St_G(m)$, $m \in \mathbb{N}$.  In the setting of
Proposition~\ref{pro:submodules-of-Wm}, it suffices to identify those
submodules $V_\mathbf{j}$ of $W_m$ that arise as
$N \St_S(m+1)/\St_S(m+1)$ for $N \trianglelefteq G$.  We observe that
there is a natural isomorphism
$N/\St_G(m+1) \cong N \St_S(m+1)/\St_S(m+1)$, when
$\St_G(m+1) \subseteq N $.  Accordingly, we set
\begin{multline*}
  R_m = \{ \mathbf{j} \in J_m \mid \exists N \trianglelefteq G :
  \St_G(m+1) \subseteq N \subseteq \St_G(m) \; \text{and} \; N /
  \St_G(m+1) \cong V_\mathbf{j} \\ \text{via the natural and the
    $\mathbb{F}_pG$-module isomorphism in \eqref{equ:iso-Wm}} \}.
\end{multline*}
Since the submodules of $W_m$ form a chain, $R_m$ is closed under
taking predecessors and hence it remains to work out $\lvert R_m
\rvert = \log_p \lvert \St_G(m) : \St_G(m+1) \rvert$. These numbers have already been worked out in \cite{FeGuTh26} for branch multi-EGS groups and in \cite{FeZu14} for GGS-groups. Here, with less effort, we arrive at the following.

\begin{proposition}\label{prop:Rm-EGS}
  For every multi-EGS group $G$ the following hold:
  \begin{enumerate}
  \item $R_m \subseteq R_{m-1} \times \{ 1,2, \dots, p \}$ and hence
    $\lvert R_m \rvert \le p \lvert R_{m-1} \rvert$ for $m \ge 2$;
  \item if $G$ is non-torsion and regular branch over $[G,G]$, then
    $\lvert R_m \rvert \ge (p-1) p^{m-1}$ and hence
    $R_m \supseteq \{1,2, \dots, p-1 \} \times \{ 1,2, \dots, p
    \}^{m-1}$ for $m \ge 1$;
  \item if $G$ is a non-torsion GGS-group and regular branch over
    $[G,G]$, then
    \[
      R_1 = \{(1),(2), \dots, (p) \} \quad \text{and} \quad R_m =
        \{1,2, \dots, p-1 \} \times \{ 1,2, \dots, p \}^{m-1} \;
        \text{for $m \ge 2$.}
    \]
  \end{enumerate}
\end{proposition}

\begin{proof}
  We just saw that
  $\lvert R_m \rvert = \log_p \lvert \St_G(m) : \St_G(m+1) \rvert$
  already determines $R_m \subseteq J_m$.

  \smallskip
  
  (i) Let $m \ge 2$ and $\mathbf{j} \in R_m$.  Choose
  $N \trianglelefteq G$ with
  $\St_G(m+1) \subseteq N \subseteq \St_G(m)$ such that
  $N / \St_G(m+1) \cong V_\mathbf{j}$.  Let $u$ denote the leftmost
  vertex at level~$1$.  Then $N$ projects under $\varphi_u$ to a
  normal subgroup $M \cong \varphi_u(N)$ of
  $\varphi_u(\St_G(1)) \cong G$ satisfying
  $\St_G(m) \subseteq M \subseteq \St_G(m-1)$ and
  $M / \St_G(m) \cong V_{\mathbf{j}'}$.  This shows that
  $\mathbf{j}' \in R_{m-1}$.

  \smallskip

  (ii) Suppose that $G$ is non-torsion and regular branch over
  $G'= [G,G]$, and let $m \in \mathbb{N}$.  From
  $\log_p \lvert G' : \St_G(2) \rvert = p-1$ and
  $G' \times \overset{p^{m-1}}{\dots} \times G' \subseteq
  \psi_{m-1}(\St_G(m))$ we conclude that
  $\lvert R_m \rvert = \log_p \lvert \St_G(m) : \St_G(m+1) \rvert \ge
  (p-1) p^{m-1}$.

  \smallskip

  (iii) Let $G$ be a non-torsion GGS-group and regular branch over
  $G'= [G,G]$.  Then
  $\lvert R_1 \rvert = \log_p \lvert \St_G(1) : \St_G(2) \rvert = p$,
  and Proposition~\ref{prop:GGS-rank-p} gives
  \[
    \lvert R_2 \rvert = \log_p \lvert \St_G(2) : \St_G(3) \rvert = p
    \log_p \lvert G' : \St_G(2) \rvert =
    (p-1)p.
  \]
  This yields the claim for $m \in \{1,2\}$.  For
    $m \ge 3$ we proceed by induction, using (i) and (ii).
\end{proof}

\begin{remark}
 Referring to part (ii) of Proposition~\ref{prop:Rm-EGS} above, depending on the group~$G$, the set $R_m$ can be strictly bigger than $ \{1,2, \dots, p-1 \} \times \{ 1,2, \dots, p
    \}^{m-1}$. Indeed, let $m=2$ and consider for example $G=\langle a,b,c\rangle$ with $\psi(b)=(a,1,\dots,1,b)$ and $\psi(c)=(c,a,a,1,\dots,1)$. Then 
    \[
    \psi(c^{a^{-1}}b^{-1}(b^a)^{-1})=(b^{-1},1,\dots,1,cb^{-1})\in\psi(\St_G(2))
    \]
    corresponds, modulo $\psi(\St_G(3))$, to an element in $V_{(p,p-1)}\smallsetminus V_{(p,p-2)}$, because 
    \begin{align*}
       \psi(b^{-1})&\equiv (a^{-1} ,1,\dots,1)\qquad\mathrm{modulo}\quad {\psi(\St_G(2))}
    \end{align*}
     and 
       \begin{align*}
       \psi(cb^{-1})&\equiv (a^{-1} ,a,a,1,\dots,1)\qquad\mathrm{modulo}\quad {\psi(\St_G(2))}.
    \end{align*}
This shows that $R_2= \{1,2, \dots, p \}^2 \smallsetminus \{(p,p)\}$. Indeed, since the total $a$-exponent of all $p^2$ components of any element of $\psi_2(\St_G(2))$ is zero, it follows that $V_{(p,p-1)}$ is the maximum possible module arising from a normal subgroup between $\St_G(2)$ and $\St_G(3)$.
\end{remark}

\begin{proof}[Proof of Theorem~\ref{thm:strong-sandwich-gives-chain}]
  The theorem is a direct consequence of Propositions
  \ref{pro:submodules-of-Wm} and~\ref{prop:Rm-EGS}, plus noting from \cite[Thm.~1.4]{ThUr21} that $\St_G(n)$ is a characteristic subgroup of $G$, for all $n\in\mathbb{N}$. 
\end{proof}

We remark that all the results of this section concerning multi-EGS groups also hold for the more general path groups~\cite{FeGuTh26}, which are  defined analogous to the multi-EGS groups but allowing for arbitrary paths.



\section{Normal generation and central width}\label{sec:width}

In this section we derive Corollaries~\ref{cor:normal-generation} and
\ref{cor:central-width}.

\begin{proof}[Proof of Corollary~\ref{cor:normal-generation}]
  Recall that $G$ is a non-torsion multi-EGS group with the congruence
  subgroup property.  Let $N \trianglelefteq G$ be a non-trivial
  normal subgroup, and let $m \in \mathbb{N}_0$ be maximal subject to
  $N \subseteq \St_G(m)$.  Theorem~\ref{thm:sharper-CSP} yields
  $\St_G(m+d) \subseteq [N,G]$ for $d \in \mathbb{N}$ depending on
  additional properties of~$G$.  Inspection of the various cases
  yields that it suffices to show that $d_G^\trianglelefteq(N) \le
  d$.  Clearly, $d_G^\trianglelefteq(N)$ equals $d(N/[N,G])$, the
  minimal number of generators of $N/[N,G]$.

  \smallskip

  Since $G$ is non-torsion, Proposition~\ref{pro:submodules-of-Wm}
  implies that each of the $d$ sections
  \[
    \big( \St_N(m+k-1) \St_G(m + k) \big) \;/\; \big(
    \St_{[N,G]}(m+k-1) \St_G(m + k) \big), \qquad k \in
    \{1,2,\dots,d\},
  \]
  constitutes a cyclic $\mathbb{F}_pG$-module.  Hence $N$ is normally
  generated by $d$ elements.

  In the case that $G$ is the Fabrykowski--Gupta group, we just showed
  $\rk^\trianglelefteq(G) \le 2$.  Since $G$ itself requires $2$
  generators (also as a normal subgroup), we deduce that
  $\rk^\trianglelefteq(G) = 2$.
\end{proof}

\begin{proof}[Proof of Corollary~\ref{cor:central-width}]
  Recall that $G$ is a non-torsion multi-EGS group and that
  $\overline{G}$ denotes the congruence completion of~$G$.  Every open
  normal subgroup of $\overline{G}$ arises as the closure
  $\overline{N}$ of a corresponding normal congruence subgroup
  $N \trianglelefteq G$.  Moreover,
  $[\overline{N},\overline{G}] = \overline{[N,G]}$ and hence it
  suffices to bound $\log_p \lvert N : [N,G] \rvert$.

  Let $N \trianglelefteq G$ be a normal congruence subgroup, and let
  $m \in \mathbb{N}_0$ be maximal subject to $N \subseteq \St_G(m)$.
  Theorem~\ref{thm:sharper-CSP} yields $\St_G(m+d) \subseteq [N,G]$
  for $d \in \mathbb{N}$ depending on additional properties of~$G$.
  Inspection of the various cases yields that it suffices to show that
  $\log_p \lvert N : [N,G] \rvert \le d$.

  \smallskip

  Since $G$ is non-torsion, Proposition~\ref{pro:submodules-of-Wm}
  implies that, for $k \in
    \{1,2,\dots,d\}$,
  \[
    i_k = \left\lvert \big( \St_N(m+k-1) \St_G(m + k) \big) : \big(
    \St_{[N,G]}(m+k-1) \St_G(m + k) \big) \right\rvert \le p.
  \]
  This implies
  $\log_p \lvert N : [N,G] \rvert = \sum_{k=1}^d \log_p i_k \le d$.

  In the case that $G$ is the Fabrykowski--Gupta group, we just showed
  $w_\mathrm{cen}(\overline{G}) \le 2$.  Since
  $\lvert G : \gamma_2(G) \rvert = p^2$, it follows that
  $w_\mathrm{cen}(\overline{G}) = 2$.
\end{proof}

We remark that our method applied to the Grigorchuk group $G$ gives
$d_G^\trianglelefteq(N) \le 3$ for every normal subgroup
$N \trianglelefteq G$ and $w_\mathrm{cen}(\overline{G}) = 3$.  This
confirms Bartholdi's conclusion from his more detailed analysis of
normal subgroups of the Grigorchuk group; see \cite[Cor.~5.2]{Ba05}.

\smallskip


 We finish with the proof of Corollary~\ref{cor:EGS-profinitely-non-isomorphic}.

\smallskip

\begin{proof}[Proof of Corollary~\ref{cor:EGS-profinitely-non-isomorphic}]
  Let $G$ and $H$ be non-torsion multi-EGS groups with the congruence
  subgroup property, such that $r_G<r_H$, and suppose that $G$ and $H$
  are regular branch over their respective derived subgroups. For each
  odd prime $p$, we have at least $\binom{p-1}{2}$ such pairs $G,H$ of
  multi-EGS groups (cf.\ Proposition~\ref{prop:EGS-CSP}(i)), and by
  Corollary~\ref{cor:central-width} their completions
  $\overline{G},\overline{H}$ have finite central width. Since the
  abelianisations $\overline{G}/[\overline{G},\overline{G}]$ and
  $\overline{H}/[\overline{H},\overline{H}]$ have different ranks, the
  result follows.
\end{proof}

Note that the lower bound $\binom{p-1}{2}$ in the proof
  above is sharp for $p=3$ (in fact, there are only two such groups)
but otherwise it is far from being optimal for larger primes.


\appendix
\section{Analogous results for the \v{S}uni\'{c} groups}\label{sec:Sunic} 

For this final part, let $p$ be any prime,  including the
  possibility $p=2$, and let $T$ denote the $p$-adic
tree.  For $r\in\mathbb{N}$ let
$f(x)=x^r+\alpha_{r-1}x^{r-1}+ \dots +\alpha_1x+\alpha_0$ be a
polynomial over~$\mathbb{F}_p$ with $\alpha_0 \ne 0$. The
\emph{\v{S}uni\'{c} group} $G=G_{p,f}$ is generated by the rooted
automorphism $a$ corresponding to the $p$-cycle
$(1 \, 2 \, \cdots \, p)\in \Sym(p)$, and by the $r$ directed
generators $b_1,\dots,b_r$ defined as follows:
\begin{align*}
  \psi(b_1)%
  & = (1,\dots,1,b_2), \quad \psi(b_2) =(1,\dots,1,b_3), \quad
    \dots, \quad \psi(b_{r-1}) =(1,\dots,1,b_{r}),\\
  \psi(b_{r})%
  & = (a,1,\dots,1,b_1^{-\alpha_0}b_2^{-\alpha_1}\cdots b_{r}^{-\alpha_{r-1}}).
\end{align*}
Similar to the previous setting, we refer to $\{a,b_1,\dots,b_r\}$ as
a standard generating system for~$G$ and we use this notation for the
generators without specific mention.  Below we collect certain facts
about \v{S}uni\'{c} groups; for more information, we refer
to~\cite{Su07,FrGa18}.  As before, $G' = [G,G]$ and $G'' = [G',G']$
denote the first and the second derived subgroups of $G$.

\begin{proposition}\cite[Lem.~1 and 6]{Su07}\label{prop:Sunic}
  Let $G=G_{p,f}$ be a \v{S}uni\'{c} group.
  \begin{enumerate}
  \item[\textup{(i)}] If $p$ is odd, then $G$ is regular branch
  over~$G'$.
  \item[\textup{(ii)}] If $p=2$ and $r = \deg(f) \ge 2$,
    then $G$ is regular branch over
    $K=\langle [a,b_2],\dots,[a,b_r]\rangle^G$.
  \end{enumerate}
\end{proposition}

The case $(p,r) = (2,1)$ yields an infinite dihedral group, which is
not regular branch.

\begin{proposition}\label{prop:Sunic-subdirect}\cite[Proof of Lem.~3.3]{FrGa18}
  Let $G=G_{p,f}$ be a \v{S}uni\'{c} group with $p$ odd.  Then
  $\psi(G')$ is subdirect in $G\times\overset{p}\dots\times G$, and
  $G$ is super strongly fractal.
\end{proposition}

Next we explain that the final statement in the above
result also holds for $p=2$.

\begin{proposition}\label{prop:Sunic-super-strongly-fractal-2}
  Let $G=G_{2,f}$ be a \v{S}uni\'{c} group with
$r = \deg(f) \ge 2$.  Then $G$ is super strongly
  fractal.
\end{proposition}

\begin{proof}
  Since $G$ is spherically transitive, for every $n\in\mathbb{N}$ it
  suffices to show that $\varphi_v(\St_G(n))=G$ for one $n$th-level
  vertex~$v$. For $n=1$, this is immediate. For $n=2$, we have that
  $b_1,\dots,b_{r-2} \in \St_G(2)$, and therefore
  $b_3,\dots,b_r\in
  \varphi_{22}(\St_G(2))$. Also
  $b_{r-1}^{\, b_r^{\, a}}\in\St_G(2)$, so
  $a\in \varphi_{22}(\St_G(2))$, and of course $b_{r-1}\in \St_G(2)$
  which gives
  $b_1^{\, \alpha_0} b_2^{\, \alpha_1} \in \varphi_{22}(\St_G(2))$
  since $b_3,\dots,b_r\in \varphi_{22}(\St_G(2))$. Next, from
  $[a,b_r]^2\in\St_G(2)$ we obtain
 \[
   \varphi_{22}([a,b_r]^2) =
   \begin{cases}
     b_2^{\, \alpha_0} b_3^{\, \alpha_1} \cdots b_{r}^{\,
       \alpha_{r-2}}%
     & \text{if $\alpha_{r-1} = 0$,} \\
     b_2^{\, \alpha_0} b_3^{\, \alpha_1} \cdots b_{r}^{\,
       \alpha_{r-2}} (b_1^{\, \alpha_0} b_2^{\, \alpha_1} \cdots
     b_{r}^{\, \alpha_{r-1}})%
     & \text{if $\alpha_{r-1}=1$.}
   \end{cases}
 \]
 Since
 $b_1^{\, \alpha_0} b_2^{\, \alpha_1} \in \varphi_{22}(\St_G(2))$ and
 $b_3, \dots, b_r\in \varphi_{22}(\St_G(2))$, plus recalling that
 $\alpha_0=1$, we get $b_1,b_2\in \varphi_{22}(\St_G(2))$ and hence
 $\varphi_{22}(\St_G(2))=G$. For $n\ge 3$, the result follows using
 the fact that $G$ is regular branch over~$K$; see
   Proposition~\ref{prop:Sunic}. Specifically, to establish
 that $\varphi_{2\,\overset{n}\dots\, 2}(\St_G(n))=G$ we
 use the elements
 \[
   \psi_{n-1}^{\, -1} \bigl( (1,\overset{2^{n-1}-1}{\dots},1,[a,b_i])
   \bigr), \quad \text{ for $i\in\{2,\dots,r\}$,}
 \]
 together with
 \[
   \psi_{n-2}^{\, -1} \bigl(
   (1,\overset{2^{n-2}-1}{\dots},1,[a,b_{r-1}]^{[b_r,a]}) \bigr)
   \qquad \text{and} \qquad \psi_{n-2}^{\, -1} \bigl(
   (1,\overset{2^{n-2}-1}{\dots},1,[b_r,a]^2) \bigr). \qedhere
 \]
\end{proof}

For notational convenience, for $n\ge 2$ we write
$p^{\underline{n}}$ for the vertex
$p\overset{n}{\dots} p$ of the tree $T$.

\begin{proposition}\label{prop:Sunic-CSP}\cite[Lem.~3.4 and
  3.6]{FrGa18} Let $G=G_{p,f}$ be a \v{S}uni\'{c} group
  and $r = \deg(f)$.
  \begin{enumerate}
  \item[\textup{(i)}] If $p$ is odd, then
   $\St_G(r+3)\subseteq G''$.
  \item[\textup{(ii)}] If $p=2$ and $r\ge 2$, then
    $\St_G(r+n_G+2)\subseteq K'=[K,K]$, where $n=n_G$ is
      such that
      $\langle a, b_1,\dots,b_{r-1}\rangle\subseteq
      \varphi_{2^{\underline{n}}}(\mathrm{st}_K(2^n))$ for
      $K=\langle [a,b_2],\dots,[a,b_r]\rangle^G$.
  \end{enumerate}
  In particular $G$ has the congruence subgroup property.
\end{proposition}

\subsection{An effective version of the congruence subgroup property}
Analogous to Proposition~\ref{prop:GGS}(i), we have the following
result for \v{S}uni\'{c} groups $G_{p,f}$ with $p$ odd.

\begin{proposition}\label{prop:Sunic-strong-CSP}
  Let $G=G_{p,f}$ be a \v{S}uni\'{c} group with $p$ odd
  and let $r = \deg(f)$.  Let $N\trianglelefteq G$ be a
 non-trivial normal subgroup and
$m\in \mathbb{N}_0$ maximal such that $N\subseteq \St_G(m)$. Then
  $$G''\times\overset{p^m}\dots\times G'' \subseteq \psi_m([N,G])$$
  and in particular $\St_G(m+r+3)\subseteq [N,G]$.
\end{proposition}

Recall from Proposition~\ref{prop:Sunic} that
  \v{S}uni\'{c} groups acting on the $2$-adic tree are
  typically regular branch. From
  Remark~\ref{rmk:CSP-bound}, we immediately obtain the following.

\begin{proposition}\label{prop:Sunic-2-CSP}
  Let $G=G_{2,f}$ be a regular branch \v{S}uni\'{c} group
 such that $r = \deg(f) \ge 2$.  Let $N\trianglelefteq G$ be a
 non-trivial normal subgroup and
  $m\in \mathbb{N}_0$ maximal such that $N\subseteq \St_G(m)$.  Then
  \[
    K''\times\overset{2^{m+1}}\dots\times K'' \subseteq
    \psi_{m+1}(\St_{[N,G]}(m+1))
  \]
  and in particular $\St_G(m+n+r+3)\subseteq [N,G]$, where
 $K=\langle [a,b_2],\dots,[a,b_r]\rangle^G$ and
    $n = n_G$ is such that
  $\langle a, b_1,\dots,b_{r-1}\rangle\subseteq
    \varphi_{2^{\underline{n}}}(\mathrm{st}_K(2^n))$.
\end{proposition}

\subsection{Normal subgroups}

In the following we make use of the notation set up in
Section~\ref{sec:chain}.

\begin{lemma}\label{lem:general-action-of-b-Sunic}
 Let $G = G_{p,f}$ be a \v{S}uni\'{c} group and let
    $r = \deg(f)$.  For $m \in \mathbb{N}$ and
  $\mathbf{j} \in J_m$ with $\mathbf{j} \ne (1,\dots,1)$,
  let
  $v \in V_{\mathbf{j} \boxplus (1)} \smallsetminus V_{\mathbf{j}^-
    \boxplus (p)}$.  Then there is an element
  $c\in \langle b_1,\dots,b_r\rangle$ such that
  \[
    v(c-1) \in V_{\mathbf{j}^- \boxplus (p)}\smallsetminus
    V_{\mathbf{j}^- \boxplus (p-1)}.
  \]
\end{lemma}

\begin{proof}
  Modulo $V_{\mathbf{j}^- \boxplus (p)}$, we may write
  $v \in V_{\mathbf{j} \boxplus (1)} \le V_{\mathbf{j}} \vert^{\oplus
    p}_\psi$ as $(v_1,\dots,v_1)$ with
  $v_1 \in V_\mathbf{j} \smallsetminus V_{\mathbf{j}^-}$.  From
  $V_{\mathbf{j}^- \boxplus (p)} (a-1) \subseteq V_{\mathbf{j}^-
    \boxplus (p-1)}$ and $(v_1,\dots,v_1) (a-1) = (0,\dots,0)$ we
  deduce that $v(a-1) \in V_{\mathbf{j}^- \boxplus (p-1)}$.  Similarly
  $V_{\mathbf{j}^- \boxplus (p)} (b-1) \subseteq V_{\mathbf{j}^-
    \boxplus (p-1)}$ for all $b\in \langle b_1,\dots,b_r\rangle$
and, in particular, we deduce that
  \[
    v(b_r-1) \equiv \big( v_1 (a-1), 1,\dots, 1, v_1
    (b_1^{-\alpha_0}\cdots b_r^{-\alpha_{r-1}}-1) \big) \qquad
    \mathrm{modulo} \quad V_{\mathbf{j}^- \boxplus (p-1)}.
  \]
  We write $\mathbf{j}=(j_1,\dots,j_m)$  and
    let $k \in \{ 1, \ldots, m \}$ be maximal with $j_k > 1$ so that
    \[
      \mathbf{j} = \mathbf{i} \boxplus (1,\overset{m-k}{\dots},1)
      \qquad \text{with} \quad \mathbf{i} = (j_1, \dots, j_k).
    \]
    If $k=m$, that is $j_m>1$, then
  $v_1 (a-1)\in V_{\mathbf{j}^- } \smallsetminus V_{\mathbf{j}^{--} }$
  and
  $v_1 (b_1^{-\alpha_0}\cdots b_r^{-\alpha_{r-1}}-1) \in
  V_{\mathbf{j}^{--}}$ imply that
  $ v(b_r-1)\in V_{\mathbf{j}^- \boxplus (p)} \smallsetminus
  V_{\mathbf{j}^- \boxplus (p-1)} $.
  
  Suppose now that $1 \le k < m$.  If $G$ is
  non-torsion, the equivalence `(i) $\Leftrightarrow$ (iii)' in
  \cite[Prop.~9]{Su07}  and \cite[Def.~2]{Su07} show that
  there is an element $b\in \langle b_1,\dots,b_r\rangle$ such that
  $\varphi_{p\overset{\,\ell-1\,}{\dots}p1}(b)\ne 1$ for all
  $\ell\in\mathbb{N}$ and, as in the proof of
  Proposition~\ref{pro:submodules-of-Wm}, we conclude
    that
  \[
    v (b-1)\in V_{\mathbf{j}^- \boxplus (p)}\smallsetminus
    V_{\mathbf{j}^- \boxplus (p-1)}.
  \]
It remains to consider the case when $G$ is a torsion
  group.  We observe that
  \[
    \mathbf{j}^- = \mathbf{i}^- \boxplus (p,\overset{m-k}\dots,p),
    \qquad \text{where} \quad\mathbf{i}^- =
    (j_1,\dots,j_{k-1},j_k-1),
  \]
  and that $v_1(a-1)\in V_{\mathbf{j}^{--}}$, as in the proof of
  Proposition~\ref{pro:submodules-of-Wm}.  Recursively, we supplement
  the original elements $v$ and $v_1$ by a sequence
  \[
    v_i\in V_{\mathbf{i} \boxplus (1,\overset{m-k+1-i}\ldots,1)}
    \smallsetminus
    V_{\mathbf{i}^- \boxplus (p,\overset{m-k+1-i}\ldots,p)},
    \quad \text{$2 \le i \le m-k+1$,}
  \]
  such that, for $1 \le i \le m-k$, the elements $v_i$ and
  $(v_{i+1},\ldots,v_{i+1})$ are congruent modulo
  $V_{\mathbf{i}^- \boxplus (p,\overset{m-k+1-i}\ldots,p)}$ and, in particular,
  \[
    v_i(a-1) \in V_{\mathbf{i}^- \boxplus
      (p,\overset{m-k-i}\ldots,p,p-1)} \qquad \text{for
      $1 \le i \le m-k$.}
  \]
  
  Moreover, we observe that
  \[
    v_{m-k+1}(a-1) \in V_{\mathbf{i}^-}\smallsetminus
    V_{\mathbf{i}^{--}}, \qquad  v_{m-k+1}(b-1) \in
    V_{\mathbf{i}^{--}} \quad \text{for $b \in \langle b_1, \dots, b_r \rangle$} 
  \]
  and
  \[
    v_i(\varphi_{p^{\underline{i}}}(b_r)-1) \in V_{\mathbf{i}^-
      \boxplus (p,\overset{m-k+1-i}\ldots,p)}\qquad \text{for
      $1 \le i \le m-k$.}
  \]
  Thus, if $\varphi_{p\overset{\,m-k\,}\dots p1}(b_r) \ne 1$ generates
  $\langle a \rangle$, we deduce from the recursive description of
  $b_r$ that
  \[
    v(b_r-1)\in V_{\mathbf{j}^- \boxplus (p)}\smallsetminus
    V_{\mathbf{j}^- \boxplus (p-1)}.
  \]
  Finally, we suppose that
  $\varphi_{p\overset{\,m-k\,}\dots p1}(b_r) =1$.  To conclude the
  proof it suffices to produce a $j \in \{1, \ldots, r-1\}$ such that
  $\varphi_{p\overset{\,m-k\,}\dots p1}(b_j) \ne 1$, for then $b_j$
  can be used in place of~$b_r$ in the previous argument.

  From $\varphi_1(b_r)=a \ne 1$ and
  $\varphi_{p\overset{\,m-k\,}\dots p1}(b_r) =1$ we deduce that there
  is a largest integer  $\eta$ such that  
 \[
   1 \le \eta \le m-k, \qquad \varphi_{p\overset{\,\eta-1\,}\dots
     p1}(b_r) \in \langle a \rangle \smallsetminus \{ 1 \}
   \qquad\text{and}\qquad \varphi_{p\overset{\,\eta\,}\dots
     p1\,}(b_r)=1.
  \]
  It is a general feature of \v{S}uni\'{c} groups that
  $\varphi_{p\overset{\,\ell\,}\dots p1}(b_r)\ne 1$ for infinitely
  many $\ell \in \mathbb{N}$.  Let $\mu \ge \eta$ be such that
  \[
    \varphi_{p\overset{\,\eta\,}\dots
      p1}(b_r)=\varphi_{p\overset{\,\eta+1\,}\dots p1}(b_r) = \dots =
    \varphi_{p\overset{\,\mu\,}\dots
      p1}(b_r)=1\qquad\text{and}\qquad
    \varphi_{p\overset{\,\mu+1\,}\dots p1}(b_r)\ne 1.
  \]
  From our set-up we see that each of the elements
  \[
    x_0 = \varphi_{p\overset{\,\eta\,}\dots p}(b_r), \quad x_1 =
    \varphi_p(x_0) = \varphi_{p\overset{\,\eta+1\,}\dots p}(b_r),\quad
    \dots, \quad x_{\mu-\eta} = \varphi_p (x_{\mu-\eta-1}) =
    \varphi_{p\overset{\,\mu\,}\dots p}(b_r)
  \]
  lies in $\langle b_1,\ldots,b_{r-1}\rangle$, whereas
  $\varphi_p (x_{\mu-\eta}) = \varphi_{p\overset{\,\mu+1\,}\dots
    p}(b_r) \in \langle b_1,\ldots,b_{r}\rangle\smallsetminus \langle
  b_1,\ldots,b_{r-1}\rangle$.  Since $\varphi_p(b_j) = b_{j+1}$ for
  $j\in\{1,\dots,r-1\}$, it follows that
  \[
    x_\ell \in \langle b_1, \dots, b_{(r-1)-(\mu-\eta)+\ell} \rangle
    \smallsetminus \langle b_1, \dots,
    b_{(r-2)-(\mu-\eta)+\ell}\rangle \qquad \text{for
      $0 \le \ell \le \mu-\eta$,}
  \]
  and, taking $\ell=0$, we deduce that $\mu-\eta \le r-2$.  Using
  $ \varphi_{p\overset{\,m-k\,}\dots p1}(b_r)= 1$ we conclude that
  $(m-k)-\eta \le \mu - \eta \le r-2$.  This shows that
  $j = (r-1) - (m-k) +\eta \in \{1, \ldots, r-1\}$ satisfies the
  requirement
  \[
    \varphi_{p\overset{\,m-k\,}\dots p1}(b_j) =
    \varphi_{p\overset{\,m-k\,}\dots p1}(b_{r-((m-k)-(\eta-1))}) =
    \varphi_{p\overset{\,\eta-1\,}\dots p1}(b_r) \ne 1. \qedhere
  \]
\end{proof}

Using Lemma~\ref{lem:general-action-of-b-Sunic} and an
  argument similar to the proof of
  Proposition~\ref{pro:submodules-of-Wm}, we obtain the following
consequence.

\begin{proposition} 
  Let $G = G_{p,f}$ be a \v{S}uni\'{c} group, and let
    $m \in \mathbb{N}$. Then the modules $V_\mathbf{j}$,
  $\mathbf{j} \in J_m$, are precisely the non-trivial submodules of
  the $\mathbb{F}_pG$-module $W_m$.

  They form a descending chain, with $V_{\mathbf{j}^-}$ being the
  unique maximal submodule of $V_\mathbf{j}$ and
  $V_\mathbf{j}/V_{\mathbf{j}^-} \cong W_0 = \mathbb{F}_p$ for each
  $\mathbf{j} \in J_m$.  In particular, every
  $\mathbb{F}_pG$-submodule of $W_m$ is cyclic.
\end{proposition}

Finally, although the precise values for $R_m$ are
  already known for the \v{S}uni\'{c} groups from \cite[Lem.~8(b) and
Cor.~1]{Su07}, it is worth noting the following
straightforward analogue of Proposition~\ref{prop:Rm-EGS}
for the \v{S}uni\'{c} groups.

\begin{proposition}\label{prop:Rm-Sunic}
  Let $G = G_{p,f}$ be a regular branch \v{S}uni\'{c}
    group and let $r = \deg(f)$.  Then $R_m = \{ 1,2, \dots, p \}^m$
  for $1 \le m \le r$ and
  \[
    R_m\supseteq\{1,2, \dots, p-1 \} \times \{ 1,2, \dots, p \}^{m-1}
    \qquad \text{for $m >r$}.
  \]
\end{proposition}

\begin{proof}
  For $r=1$ the claim follows as for
  Proposition~\ref{prop:Rm-EGS}(iii), so we suppose that
  $r\ge 2$. We have $b_{r-i}\in \St_G(i+1)$ for
  $i\in\{1,\dots,r-1\}$. This yields the first part of the statement.
  For $m>r$, the result follows similarly from considering the element
  \[
    \psi_{m-1}^{\, -1}([a,b_r],1,\dots,1)\in\St_G(m).\qedhere
  \]
\end{proof}

We remark that data from \cite{Su07} shows that the final containment
in the above proposition is in most cases strict.

\smallskip

With only minor modifications, corresponding statements to
Theorems~\ref{thm:sharper-CSP}, \ref{thm:strong-sandwich-gives-chain}
and \ref{thm:Sunic} can be proved for the Grigorchuk groups acting on
the binary rooted tree.





\end{document}